\documentclass [12pt,a4paper,reqno]{amsart}
\input amssymb.sty

\usepackage{color}
\usepackage{array}
\usepackage{tabularx}
\usepackage{enumerate}
\usepackage{amsmath}
\usepackage{amsfonts}
\usepackage{amssymb}
\usepackage{bbm}
\usepackage{mathrsfs}
\usepackage{latexsym}
\usepackage{epsfig}

\newtheorem{definition}{Definition}[section]

\newtheorem{lemma}[definition]{Lemma}
\newtheorem{theorem}[definition]{Theorem}
\newtheorem{proposition}[definition]{Proposition}
\newtheorem{corollary}[definition]{Corollary}
\newtheorem{example}[definition]{Example}
\newtheorem{remark}[definition]{Remark}

\begin{document}

%%%%%%%%%%%%%%%%%%%%%%%%%%%%%%%%%%%%%%%%%%%%%%%%%%%%%%%%%%%%%%%%%%%%%%%%%%%%%%%
%%% quelques definitions utiles
%\newcommand{\proof}{\paragraph{Proof}}
\newcommand{\fin}{$\Box$\\}
\newcommand{\C}{\mathbf{C}}
\newcommand{\R}{\mathbf{R}}
\newcommand{\Q}{\mathbf{Q}}
\newcommand{\Z}{\mathbf{Z}}
\newcommand{\N}{\mathbf{N}}
\newcommand{\ds}{\displaystyle}
\newcommand{\saut}[1]{\hfill\\[#1]}
\newcommand{\vsp}{\vspace{.15cm}}
\newcommand{\difrac}{\displaystyle \frac}
\newcommand{\dist}{\textrm{dist}}
\newcommand{\diam}{\mathrm{diam}\ }
\newcommand{\var}{\mathrm{Var}}
\newcommand{\mbf}{\textbf}
\newcommand{\levy}{\mathscr{B}}
\newcommand{\sheet}{\mathbbm{B}}
\newcommand{\sifbm}{\mathbf{B}}
\newcommand{\alphar}{\texttt{\large $\boldsymbol{\alpha}$}}
%%%%%%%%%%%%%%%%%%%%%%%%%%%%%%%%%%%%%%%%%%%%%%%%%%%%%%%%%%%%%%%%%%%%%%%%%%%%%%%
\title[The set-indexed L\'evy process]{The set-indexed L\'evy process: \\Stationarity, Markov and sample paths properties}

\author{Erick Herbin}
\address{Ecole Centrale Paris, Grande Voie des Vignes, 92295 Chatenay-Malabry, France} \email{erick.herbin@ecp.fr}
\author{Ely Merzbach}
\address{Dept. of Mathematics,
Bar Ilan University, 52900 Ramat-Gan, Israel}\email{merzbach@macs.biu.ac.il}

\subjclass[2000]{60\,G\,10, 60\,G\,15, 60\,G\,17, 60\,G\,18, 60\,G\,51, 60\,G\,60}

\begin{abstract}
We present a satisfactory definition of the important class of L\'evy processes indexed by a general collection of sets. We use a new definition for increment stationarity of set-indexed processes to obtain different characterizations of this class. As an example, the set-indexed compound Poisson process is introduced.
The set-indexed L\'evy process is characterized by infinitely divisible laws and a L\'evy-Khintchine representation.
Moreover, the following concepts are discussed: projections on flows, Markov properties, and pointwise continuity. Finally the study of sample paths leads to a L\'evy-It\^o decomposition.  As a corollary, the semimartingale property is proved.
\end{abstract}

\maketitle

\section{Introduction}

L\'evy processes constitute a very natural and a fundamental class of stochastic processes, including Brownian motion, Poisson processes and stable processes. On the other hand, set-indexed processes like the set-indexed Brownian motion (also called the white noise) and the spatial Poisson process are very important in several fields of applied probability and spatial statistics. As a general extension of these processes, the aim of this paper is to present a satisfactory definition of the notion of set-indexed L\'evy processes and to study its properties.

More precisely, the processes studied are indexed by a quite general class $\mathcal{A}$ of closed subsets in a measure space $(\mathcal{T} ,m)$. Our definition of L\'evy processes is sufficiently broad to include the set-indexed Brownian motion, the spatial Poisson process, spatial compound Poisson processes and some other stable processes. 
In the case that $\mathcal{T}$ is the $N$-dimensional rectangle $[0 ,1]^N$ and $m$ is the Lebesgue measure, a similar definition was given and studied by Vares \cite{vares}, by Bass and Pyke  \cite{BP}  and by Adler and Feigin \cite{AF}. However, in our framework the parameter set is more general, the $2^N$ quadrants associated with any point do not exist, and we do not assume artificial hypothesis. 
As it will be shown later, no group structure is needed in order to define the increment stationarity property for L\'evy processes.

As motivation, notice that our setting includes at least two other interesting cases. The first one is still the Euclidean space, but instead of considering rectangles, we can consider more general sets like the class of ``lower layer" sets. 
The second case occurs when the space is a tree and we obtain L\'evy processes indexed by the branches of the tree. We refer to \cite{WI} and \cite{T} for applications in environmental sciences and cell biology of some kinds of L\'evy processes indexed by subsets of the Euclidean space $\mathbf{R}^N$.

A related concept is the notion of infinitely divisible and independently scattered random measure (infinitely divisible ISRM, for short). The usual reference on the subject is the work of Rajput and Rosinski (\cite{RR89}, see also other references mentioned therein). 
In their paper, integral representations of stochastic processes $X$ indexed by an abstract space are studied. Their goal   is to consider representations $X_t=\int_{\mathcal{T}}f_t\ d\Lambda$ and to derive distributional and sample paths properties of $X$ from the deterministic function $f_t$ and the integrator ("noise") $\Lambda$, defined as an independently scattered random measure on a $\delta$-ring of subsets in $\mathcal{T}$. 
In contrast to this previous work, our goal is not to reduce set-indexed processes to random measures, but to endow the indexing collection with a structure allowing the
study of standard issues of stochastic processes general theory, such as martingale, Markov and sample paths properties.

More precisely, for any $\mathcal{A}$-indexed process $X=\{X_U;\;U\in\mathcal{A}\}$, we define the {\em increment process} $\Delta X$, indexed by the collection of elements $C=U_0\setminus\bigcup_{i=1}^n U_i$ with $U_i\in\mathcal{A}$ for all $i=0,\dots,n$, by $\Delta X_C = X_{U_0} -\Delta X_{U_0 \cap \bigcup U_i}$, where $\Delta X_{U_0\cap \bigcup U_i}$ is given by the {\em inclusion-exclusion} formula (see the next section). 
If the process $X$ is a set-indexed L\'evy process (Definition 3.1), then one can show that $\Delta X$ is an infinitely divisible ISRM (Theorem 4.3). 
More than the random measure point of view, our paper focuses on the increment stationarity property, which can be investigated thanks to the general framework of indexing collection, and  on the study of Markov properties, projections on increasing flows, sample paths and martingale properties for set-indexed processes  $X=\{X_U;\;U\in\mathcal{A}\}$, presented in Sections 5, 6 and 7.

In order to extend the definition of classical L\'evy process to set-indexed L\'evy process, we need the concepts of increments independence, continuity in law and stationarity of increments. The first two properties can be trivially extended to the set-indexed framework and these processes have been considered in the Euclidean space $\R^N_+$ by several authors: Adler et al. \cite{AMSW}, Adler and Feigin \cite{AF}, as well as Bass and Pyke \cite{BP} studied this type of processes, adding a measure continuity property. In \cite{Ba02}, Balan considers set-indexed processes, introducing a property of monotone continuity in probability. However the concept of increment stationarity cannot be easily extended in the set-indexed framework; so this notion was ignored by most of the authors (except in \cite{BP} in which a kind of measure stationarity is implicitely assumed). Their definitions of L\'evy processes restricted to the one-parameter case are called today additive processes. In our definition of a set-indexed L\'evy process, we require a stationarity property and it plays a fundamental role. In particular, we will prove that a set-indexed process such that its projection on every increasing path is a real-parameter L\'evy process is a set-indexed L\'evy process. Under some conditions, the converse holds too.

Among the different possibilities, is there a natural definition of stationarity increments? The key to the answer can be found in the fractional Brownian motion theory. 
In \cite{ehem} and \cite{flow}, we defined and characterized the set-indexed fractional Brownian motion on the space $(\mathcal{T},\mathcal{A},m)$. Giving precise definitions for self-similarity and increment stationarity of set-indexed processes, as in the one-parameter  case, it was proved that the set-indexed fractional Brownian motion is the only set-indexed Gaussian process which is self-similar and has $m$-stationary $\mathcal{C}_0$-increments (will be defined in the next section).
An important justification to our definition for increment stationarity is that its projection on any flow (that is an increasing function from a positive interval into $\mathcal{A}$) leads to the usual definition for increment stationary one-parameter process. More precisely, if $X$ is a $\mathcal{C}_0$-increments $m$-stationary process and $f:[a,b]\rightarrow\mathcal{A}$ a flow, then the $m$-standard projection $X^{m,f}$ of $X$ on $f$ has stationary increments:
\begin{equation}\label{stat1p}
\forall h\in\mathbf{R}_+;\quad
\left\{X^{m,f}_{t+h}-X^{m,f}_h;\;t\in [a-h,b-h]\right\} \stackrel{(d)}{=}
\left\{X^{m,f}_{t}-X^{m,f}_0;\;t\in [a,b]\right\}.
\end{equation}

This satisfactory definition of stationarity opens the door to our new definition of L\'evy process.

It is important to emphasize that the study of set-indexed L\'evy processes is not a simple extension of the classical L\'evy process. Some of the specific properties of the set-indexed L\'evy process lead to a better understanding of fundamental properties of stochastic processes; for example the measure-based definition for increment stationarity or the analysis of sample paths regularity giving rise to different types of discontinuity like point mass jumps.

In the next section, we present the general framework of set-indexed processes, and we study the basic notions of independence and stationarity in the set-indexed theory. 
In Section 3, we give the definition of set-indexed L\'evy processes and discuss some simple examples such as the set-indexed compound Poisson process. We compare our definition with the concept of ISRM. 
In Section 4, we discuss links with infinitely divisible distributions and prove the L\'evy-Khintchine representation formula. Moreover, we prove the equivalence between set-indexed L\'evy processes on the indexing collection $\mathcal{A}$ and infinitely divisible ISRM on the collection of elements $C=U_0\setminus\bigcup_{i=1}^n U_i$ with $U_i\in\mathcal{A}$ for all $i=0,\dots,n$. 
Section 5 is devoted to projections on flows. We present some characterizations of the set-indexed L\'evy process by its projections on all the different flows. 
Markov properties are the object of Section 6. It is shown that a set-indexed L\'evy process is a Markov process and conversely for any homogeneous transition system, there exists a Markov process with this transition system which is a set-indexed L\'evy process. 
Finally, sample paths properties of the set-indexed L\'evy process are analysed in the last Section 7. Pointwise continuity is defined, and we prove that the sample paths of any set-indexed L\'evy process with Gaussian increments are almost surely pointwise continuous. We obtain a L\'evy-It\^o decomposition and therefore another characterization of the set-indexed L\'evy process as the sum of a local strong martingale and a Radon measure process is proved.

\section{Framework and set-indexed increment stationarity}\label{sec:frame}

We follow \cite{IvMe00} and \cite{ehem} for the framework and notations. Our processes are indexed
by an {\em indexing collection} $\mathcal{A}$ of compact subsets of a locally compact
metric space $\mathcal{T}$ equipped with a Radon measure $m$ (denoted $(\mathcal{T},m)$).

In   the entire paper, for any class $\mathcal{D}$ of subsets of
$\mathcal{T}$, $\mathcal D(u)$ denotes the class of finite unions
of elements of $\mathcal D$.

\begin{definition}[Indexing collection]\label{basic}
A nonempty class $\mathcal{A}$ of compact, connected subsets of $\mathcal{T}$ is
called an {\em indexing collection}
if it satisfies the following:
\begin{enumerate}
 \item\label{IC:C1} $\emptyset\in\mathcal{A}$, and $A^{\circ}\neq A$
if $A\notin\left\{ \emptyset, \mathcal{T} \right\}$.
In addition, there exists an increasing sequence
$\left(B_n\right)_{n\in\mathbf{N}}$ of sets in $\mathcal{A}(u)$ such that
$\mathcal{T}=\bigcup_{n\in\mathbf{N}}B^{\circ}_n$.

\item $\mathcal{A}$ is closed under arbitrary intersections and if
$A,B\in\mathcal{A}$ are nonempty, then $A\cap B$ is nonempty.
If $(A_i)_{i\in\N}$ is an increasing sequence in $\mathcal{A}$ and if there exists
$n\in\mathbf{N}$ such that $A_i\subseteq B_n$ for all $i$,
then $\overline{\bigcup_{i\in\N} A_i}\in\mathcal{A}$.

\item The $\sigma$-algebra generated by $\mathcal{A}$, $\sigma (\mathcal{A})=\mathcal{B}$,
the collection of all Borel sets of $\mathcal{T}$.

\item {\em Separability from above}\\
There exists an increasing sequence of finite subclasses
$\mathcal{A}_n=\{A_1^n,...,A_{k_n}^n\}$ of $\mathcal{A}$ closed under intersections
and satisfying $\emptyset, B_n \in\mathcal{A}_n(u)$
and a sequence of functions
$g_n:\mathcal{A}\rightarrow\mathcal{A}_n(u)\cup\left\{\mathcal{T}\right\}$
satisfying
\begin{enumerate}
\item $g_n$ preserves arbitrary intersections and finite unions \\
(i.e. $g_n(\bigcap_{A\in\mathcal{A}'}A )=\bigcap_{A\in\mathcal{A}'}g_n(A )$
for any $\mathcal{A}'\subseteq\mathcal{A}$, and \\
if $\bigcup_{i=1}^kA_i=\bigcup_{j=1}^mA_j'$, then
$\bigcup_{i=1}^kg_n(A_i)=\bigcup_{j=1}^mg_n(A_j')$);

\item for any $A\in\mathcal{A}$, $A\subseteq (g_n(A))^{\circ}$ for all $n$;

\item $g_n(A)\subseteq g_m(A)$ if $n\geq m$;

\item for any $A\in\mathcal{A}$, $A=\bigcap_n g_n(A)$;

\item if $A,A'\in\mathcal{A}$ then for every $n$, $g_n(A) \cap A' \in\mathcal{A}$, and if
$A'\in\mathcal{A} _n$ then $g_n(A) \cap A' \in\mathcal{A} _n$;
\item   $g_n(\emptyset)=\emptyset$, for all $n\in\N$.
\end{enumerate}

\item Every countable intersection of sets in $\mathcal{A}(u)$ may be expressed as
the closure of a countable union of sets in $\mathcal{A}$.

\end{enumerate}
(Note: ` $\subset$' indicates strict inclusion; `$\overline{(\cdot)}$'
and `$(\cdot )^{\circ}$' denote respectively the closure and the interior of a set.)
\end{definition}

Let $\emptyset'=\bigcap_{U\in\mathcal{A}\setminus\{\emptyset\}} U$ be the minimal set in $\mathcal{A}$ ($\emptyset'\ne\emptyset$). The role played by $\emptyset'$ is similar to the role played by $0$ in the classical theory. We assume that $m(\emptyset')=0$.

\begin{example}
There are many examples of indexing collection that have already been deeply studied (cf. \cite{Ivanoff}, \cite{flow}). Let us mention
\begin{itemize}
\item rectangles of $\mathbf{R}^N_+$:
$\mathcal{A}=\left\{[0,t];t\in\mathbf{R}^N_+\right\}\cup\left\{\emptyset\right\}$.
In that case, any $\mathcal{A}$-indexed process can be seen as a
$N$-parameter process and conversely. \item arcs of the unit
circle $\mathbb{S}_1$ in $\mathbf{R}^2$:
$\mathcal{A}=\Big\{\overset{\curvearrowright}{0M};
M\in\mathbb{S}_1\Big\}\cup\left\{\emptyset\right\}$. In that case,
any $\mathcal{A}$-indexed process can be seen as a process indexed
by points on the circle. \item lower layers in $\mathbf{R}^N_+$:
$\mathcal{A}$ is the set of $U\subset\mathbf{R}^N_+$ such that
$t\in U \Rightarrow [0,t]\subseteq U$.
\end{itemize}
\end{example}

\

The difficulty to give a good definition for set-indexed L\'evy processes is related to increment stationarity. In this paper, we use the same definition as in \cite{flow}, and for this purpose we need first to extend the collection $\mathcal{A}$ to the following collections:
\begin{itemize}
\item The class $\mathcal{C}$ is defined as the collection of
elements $U_0\setminus \bigcup_{i=1}^n U_i$ where $U_i\in
\mathcal{A}$ for all $i=0,\dots,n$.

\item The class $\mathcal{C}_0$ is defined as the sub-class of $\mathcal{C}$ of elements $U\setminus V$ where $U,V\in\mathcal{A}$.
\end{itemize}

\noindent Since $\emptyset$ belongs to $\mathcal{A}$, we have the inclusion $\mathcal{A}\subset\mathcal{C}_0\subset\mathcal{C}$.

\begin{lemma}\label{lemstructC}
Every element $C\in\mathcal{C}$ admits a representation $C=U\setminus V$ where $U=\inf\{W\in\mathcal{A}: C\subset W\}$ and $V=\sup\{W\in\mathcal{A}(u): W\subset U\textrm{ and }W\cap C=\emptyset\}$.
\end{lemma}

\begin{proof}
Consider an element $C=U\setminus\bigcup_i U_i$ of $\mathcal{C}$ (with $U_i\subset U$ for all $i$) and $\tilde{U}=\inf\{V\in\mathcal{A}: C\subset V\}$. The set $\tilde{U}$ is the intersection of all $V\in\mathcal{A}$ such that $C\subset V$. Therefore, we have $C\subset\tilde{U}$.
By definition, we have $C\subset U$ and $\tilde{U}\subset U$.

Assume the existence of $x\in U\setminus\tilde{U}$. We have necessarily $x\notin C$ and then $x\in\bigcup_i U_i$. This shows
\begin{equation*}
U=\tilde{U} \cup \bigcup_i U_i
\end{equation*}
and then
\begin{equation*}
C = U\setminus\bigcup_i U_i = \tilde{U}\setminus\bigcup_i U_i.
\end{equation*}

It remains to prove that $\bigcup_i U_i = \sup\{W\in\mathcal{A}(u): W\cap C=\emptyset\}$. Let us define
\begin{equation*}
V=\bigcup_{\genfrac{}{}{0pt}{}{W\in\mathcal{A}(u),W\subset U}{W\cap C=\emptyset}} W.
\end{equation*}
It is clear that $\bigcup_i U_i\subset V$. For all $x\in V$, $x$
belongs to some $W\in\mathcal{A}(u)$ where $W\subset U$ and $W\cap
C=\emptyset$. Consequently, $x\in U$ and $x\notin C$. Then
$x\in\bigcup_i U_i$. This shows that $V=\bigcup_i U_i$.
\end{proof}

From any $\mathcal{A}$-indexed process $X= \{X_U;\;
U\in\mathcal{A}\}$, we define the {\em increment process} $\Delta
X = \{ \Delta X_C;\;C\in\mathcal{C} \}$ by $\Delta X_C = X_{U_0} -
\Delta X_{U_0 \cap \bigcup U_i}$ for all $C=U_0\setminus
\bigcup_{i=1}^n U_i$, where $\Delta X_{U_0 \cap \bigcup U_i}$ is
given by the {\em inclusion-exclusion} formula
\begin{equation*}
\Delta X_{U_0 \cap \bigcup U_i} = \sum_{i=1}^n \sum_{j_1<\dots<j_i} (-1)^{i-1} X_{U_0 \cap U_{j_1} \cap \dots \cap U_{j_i}}.
\end{equation*}
When $C=U\setminus V\in\mathcal{C}_0$, the expression of $\Delta X_C$ reduces to $\Delta X_C = X_U - X_{U\cap V}$.

The existence of the increment process $\Delta X$ assumes that the
value $\Delta X_C$ does not depend on the representation of $C$
and $X_{\emptyset}=\Delta X_{\emptyset} = 0$.

\

An $\mathcal{A}$-indexed process $X= \{X_U;\; U\in\mathcal{A}\}$
is said to have $m$-stationary $\mathcal{C}_0$-increments if for
all integer $n$, all $V\in\mathcal{A}$ and for all increasing
sequences $(U_i)_i$ and $(A_i)_i$ in $\mathcal{A}$, we have
\begin{equation*}
\left[ \forall i,\; m(U_i\setminus V)=m(A_i) \right]
\Rightarrow
(\Delta X_{U_1\setminus V},\dots,\Delta X_{U_n\setminus V}) \stackrel{(d)}{=}
(\Delta X_{A_1},\dots,\Delta X_{A_n}).
\end{equation*}

This definition of increment stationarity for a set-indexed process is the natural extension of increment stationarity for one-dimensional processes. It can be seen as the characteristic of a set-indexed process whose projection on any flow has stationary increments, in the usual sense for one-parameter processes (see \cite{flow}).

\

In the real-parameter setting, independence of increments allows to reduce the increment stationarity property to a simpler statement with only two increments. The following result shows that this fact remains true for set-indexed processes and that the definition of stationarity in \cite{flow} is equivalent to $\mathcal{C}_0$-increment stationarity in the previous sense of \cite{ehem}.
Consequently, Condition (i.) of Lemma \ref{lemequivstat} can be considered for increment stationarity of set-indexed L\'evy processes.

\begin{lemma}\label{lemequivstat}
Let $X=\left\{X_U;\;U\in\mathcal{A}\right\}$ be a set-indexed process satisfying the following property:
\begin{quote}
For all $C_1=U_1\setminus V_1,\dots, C_n=U_n\setminus V_n$ in $\mathcal{C}_0$ such that
\begin{align*}
\forall i=1,\dots,n;\quad &V_i\subset U_i\\
\forall i=1,\dots,n-1;\quad & U_i\subset V_{i+1},
\end{align*}
the random variables $\Delta X_{C_1},\dots,\Delta X_{C_n}$ are independent.
\end{quote}
\

\noindent Then the two following assertions are equivalent:
\begin{enumerate}[(i.)]
\item For all $C_1=U_1\setminus V_1$ and $C_2=U_2\setminus V_2$ in $\mathcal{C}_0$, we have
\begin{equation*}
m(U_1\setminus V_1)=m(U_2\setminus V_2)
\Rightarrow
\Delta X_{U_1\setminus V_1} \stackrel{(d)}{=}
\Delta X_{U_2\setminus V_2}
\end{equation*}

\item $X$ has $m$-stationary $\mathcal{C}_0$-increments, i.e. for all integer $n$, all $V\in\mathcal{A}$ and for all increasing sequences $(U_i)_i$ and $(A_i)_i$ in $\mathcal{A}$, we have
\begin{equation*}
\left[ \forall i,\; m(U_i\setminus V)=m(A_i) \right]
\Rightarrow
(\Delta X_{U_1\setminus V},\dots,\Delta X_{U_n\setminus V}) \stackrel{(d)}{=}
(\Delta X_{A_1},\dots,\Delta X_{A_n}).
\end{equation*}

\end{enumerate}
\end{lemma}

\begin{proof}
The implication $(ii.)\Rightarrow(i.)$ is obvious.

Conversely, assume that $(i.)$ holds and consider $V$, $(U_i)_i$ and $(A_i)_i$ as in $(ii.)$. The law of $(\Delta X_{U_1\setminus V},\dots,\Delta X_{U_n\setminus V})$ is determined by
$E\left[\exp\left(i\sum_{j=1}^n \lambda_j \Delta X_{U_j\setminus V} \right)\right]$, where $\lambda_1,\dots,\lambda_n\in\mathbf{R}$.

We can write
\begin{eqnarray*}
\left\{\begin{array}{l}
\Delta X_{U_2\setminus V} = \Delta X_{U_2\setminus U_1}+\Delta X_{U_1\setminus V}\\
\Delta X_{U_3\setminus V} = \Delta X_{U_3\setminus U_2}+\Delta X_{U_2\setminus U_1}+\Delta X_{U_1\setminus V}\\
\dots\\
\Delta X_{U_n\setminus V} = \Delta X_{U_n\setminus U_{n-1}}+\Delta X_{U_{n-1}\setminus U_{n-2}}+\dots+\Delta X_{U_1\setminus V}
\end{array}\right.
\end{eqnarray*}
which implies
\begin{align*}
E\left[\exp\left(i\sum_{j=1}^n \lambda_j \Delta X_{U_j\setminus V} \right)\right]
=E\left[\exp\left(i\sum_{j=1}^n(\lambda_j+\dots+\lambda_n)\Delta X_{U_j\setminus U_{j-1}}\right)\right]
\end{align*}
where $U_0=V$.
Using the independence of the r.v. $\Delta X_{U_j\setminus U_{j-1}}$, we get
\begin{align*}
E\left[\exp\left(i\sum_{j=1}^n \lambda_j \Delta X_{U_j\setminus V} \right)\right]
=\prod_{j=1}^n E\left[\exp\left(i(\lambda_j+\dots+\lambda_n)\Delta X_{U_j\setminus U_{j-1}}\right)\right].
\end{align*}
As the assertion $(i.)$ holds, we have for all $j=1,\dots,n$,
\begin{align*}
E\left[\exp\left(i(\lambda_j+\dots+\lambda_n)\Delta X_{U_j\setminus U_{j-1}}\right)\right]
=E\left[\exp\left(i(\lambda_j+\dots+\lambda_n)\Delta X_{A_j\setminus A_{j-1}}\right)\right]
\end{align*}
and then, by independence of the r.v. $\Delta X_{A_j\setminus A_{j-1}}$,
\begin{align}\label{eqequivstat}
E\left[\exp\left(i\sum_{j=1}^n \lambda_j \Delta X_{U_j\setminus V} \right)\right]
&=E\left[\exp\left(i\sum_{j=1}^n(\lambda_j+\dots+\lambda_n)\Delta X_{A_j\setminus A_{j-1}}\right)\right] \nonumber\\
&=E\left[\exp\left(i\sum_{j=1}^n \lambda_j \Delta X_{A_j} \right)\right].
\end{align}
{}From (\ref{eqequivstat}), the assertion $(ii.)$ is proved.
\end{proof}

\

\section{Definition and examples}\label{sec:def}

The independence of increments and the increment stationarity property discussed in the previous section allow to define the class of set-indexed L\'evy processes. It is shown in Example \ref{ex:Levy} that this class gathers together the classical set-indexed Brownian motion and the spatial Poisson process.

\begin{definition}\label{deflevy}
A set-indexed process $X=\left\{X_U;\;U\in\mathcal{A}\right\}$ with definite increments is called a {\em set-indexed L\'evy process} if the following conditions hold:

\begin{enumerate}
\item $X_{\emptyset'}=0$ almost surely.

\item the increments of $X$ are independent: for all pairwise disjoint $C_1,\dots,C_n$ in $\mathcal{C}$, the random variables $\Delta X_{C_1},\dots,\Delta X_{C_n}$ are independent.

\item $X$ has $m$-stationary $\mathcal{C}_0$-increments, i.e. for all integer $n$, all $V\in\mathcal{A}$ and for all increasing sequences $(U_i)_i$ and $(A_i)_i$ in $\mathcal{A}$, we have
\begin{equation*}
\left[ \forall i,\; m(U_i\setminus V)=m(A_i) \right]
\Rightarrow
(\Delta X_{U_1\setminus V},\dots,\Delta X_{U_n\setminus V}) \stackrel{(d)}{=}
(\Delta X_{A_1},\dots,\Delta X_{A_n}).
\end{equation*}

\item $X$ is continuous in probability: if $(U_n)_{n\in\N}$ is a sequence in $\mathcal{A}$ such that
\begin{equation}\label{eq:cond4}
\overline{\bigcup_{n} \bigcap_{k\geq n} U_k} = \bigcap_n
\overline{\bigcup_{k\geq n} U_k} = A \in\mathcal{A},
\end{equation}
then
\begin{equation*}
\lim_{n\rightarrow\infty} P\left\{ |X_{U_n} - X_A| > \epsilon \right\} = 0.
\end{equation*}

\end{enumerate}
\end{definition}

% $X_{\emptyset'}=0$ does not seem to be necessary, if we do not assume that m(\emptyset')=0.

Our definition of probability continuity is stronger than the definition given in \cite{Ba02}, in which only monotone continuity in probability is required. In fact our definition is very natural and is closed to the so-called Painlev\'e-Kuratowski topology, which is itself equivalent to the Fell topology for closed sets (see \cite{Molchanov} for details).

\begin{remark}
The condition $(2)$ is equivalent to:
for all pairwise disjoints $C_1,\dots,C_n$ in $\mathcal{C}(u)$, the random variables $\Delta X_{C_1},\dots,\Delta X_{C_n}$ are independent.
\end{remark}

As a corollary to Lemma \ref{lemequivstat}, we can state the equivalent following definition for set-indexed L\'evy processes:

\begin{proposition}\label{defLevyequiv}
A set-indexed process $X=\left\{X_U;\;U\in\mathcal{A}\right\}$ with definite increments is called a {\em set-indexed L\'evy process} if the following conditions hold

\begin{enumerate}
\item $X_{\emptyset'}=0$ almost surely.

\item for all pairwise disjoint $C_1,\dots,C_n$ in $\mathcal{C}$, the random variables $\Delta X_{C_1},\dots,\Delta X_{C_n}$ are independent.

\item[(3')] for all $C_1=U_1\setminus V_1$ and $C_2=U_2\setminus V_2$ in $\mathcal{C}_0$, we have
\begin{equation*}
m(U_1\setminus V_1)=m(U_2\setminus V_2)
\Rightarrow
\Delta X_{U_1\setminus V_1} \stackrel{(d)}{=}
\Delta X_{U_2\setminus V_2}.
\end{equation*}

\item[(4)] $X$ is continuous in probability: if $(U_n)_{n\in\N}$ is a sequence in $\mathcal{A}$ such that
\begin{equation*}
\overline{\bigcup_{n} \bigcap_{k\geq n} U_k} = \bigcap_n \overline{\bigcup_{k\geq n} U_k} = A \in\mathcal{A}
\end{equation*}
then
\begin{equation*}
\lim_{n\rightarrow\infty} P\left\{ |X_{U_n} - X_A| > \epsilon \right\} = 0.
\end{equation*}

\end{enumerate}
\end{proposition}

Following the formalism of Rajput and Rosinski (\cite{RR89}), since the collection $\mathcal{C}$ constitutes a $\delta$-ring of subsets of $\mathcal{T}$, i.e. closed under countable intersections, the natural question arising from Definition \ref{deflevy} is: Is the increment process $\{\Delta X_C;\;C\in\mathcal{C}\}$ of a set-indexed L\'evy process $\{X_U;\;U\in\mathcal{A}\}$ an {\em independently scattered random measure (ISRM)} (or equivalently, a {\em stochastic completely additive set function} using the terminology of Pr\'ekopa (\cite{Prekopa-I}))  on $\mathcal{C}$?

\begin{itemize}
\item From Condition (2) of Definition \ref{deflevy}, for any sequence $(C_n)_{n\in\N}$ of disjoint sets in $\mathcal{C}$, the random variables $\left(\Delta X_{C_n}\right)_{n\in\N}$ are independant (using the fact that the law of $\left(\Delta X_{C_n}\right)_{n\in\N}$ is generated by its finite-dimensional distributions).
\item The existence of the increment process $\{\Delta X_C;\;C\in\mathcal{C}\}$ is equivalent to the finite additivity of $\Delta X$.
Moreover, Condition (4) of Definition \ref{deflevy} implies that if $(C_n)_{n\in\N}$ is a non-increasing sequence in $\mathcal{C}$ such that $\lim_{n\rightarrow\infty}C_n = \emptyset$, $\Delta X_{C_n}$ converges to $0$ in probability. The monotony of $(C_n)_{n\in\N}$ allows to conclude that the convergence is almost sure.
From this, if $(C_n)_{n\in\N}$ is any sequence in $\mathcal{C}$ such that $\bigcup_{n\in\N}C_n\in\mathcal{C}$, then
$$ \Delta X_{\bigcup_{n\in\N}C_n} = \sum_{n\in\N} \Delta X_{C_n} \quad\textrm{a.s.}$$
where the sum converges almost surely.
\end{itemize}

\noindent From these two points, we can claim that the increment process $\Delta X$ of any set-indexed L\'evy process is an independently scattered random measure on $\mathcal{C}$.

\noindent However, we emphasize the importance of the stationarity property in the definition of set-indexed L\'evy processes. The sub-class $\mathcal{C}_0$ of $\mathcal{C}$ plays a fondamental role, since it is known that $m$-stationarity of $\mathcal{C}_0$-increments is equivalent to: the projections of the set-indexed process on any flow is a real-parameter process with stationary increments in the usual sense (see\cite{flow}). In the next section, we prove that this property and the continuity in probability imply infinite divisibility and allows to determine explicitely the finite-dimensional distributions of set-indexed L\'evy processes.

\noindent Moreover, instead of reducing the SI L\'evy processes to a random measure on $\mathcal{C}$, we consider the structure of the indexing collection $\mathcal{A}$ in order to define filtrations indexed by $\mathcal{A}$ and to study usual distributional properties of stochastic processes, such as Markov and martingale properties. 

\begin{example}\label{ex:Levy}
Several set-indexed processes that have been extensively studied (cf. \cite{Adler, BP, C}, \dots) satisfy Definition \ref{deflevy} of set-indexed L\'evy processes:
\begin{itemize}
\item {\em Deterministic process:} A process $X=\{X_U;\;U\in\mathcal{A}\}$ such that for all $U\in\mathcal{A}$, $X_U=c.m(U)$ for some constant $c\in\R$;

\item {\em Set-indexed Brownian motion:}
A mean-zero Gaussian set-indexed process $B=\{B_U;\;U\in\mathcal{A}\}$ such that
\begin{equation*}
\forall U,V\in\mathcal{A};\quad E[B_U B_V] = m(U\cap V)
\end{equation*}
where $m$ denote the measure of the space $(\mathcal{T},\mathcal{A},m)$.

The fact that $B_{\emptyset'}=0$ a.s. is a consequence of
$m(\emptyset')=0$. The condition (2) is well-known and the
increment stationarity follows from Proposition 5.2 in \cite{flow}
with $H=1/2$ ($B$ actually satisfies a stronger definition for
increment stationarity replacing the class $\mathcal{C}_0$ with
the Borel sets). For the stochastic continuity, consider a
sequence $(U_n)_{n\in\N}$ in $\mathcal{A}$ as in (\ref{eq:cond4}).
We have, for all $n\in\N$
\begin{align*}
E[B_{U_n}-B_A]^2 &= m(U_n) + m(A) - 2 m(U_n \cap A) \\
&= m(U_n \setminus A) + m(A \setminus U_n).
\end{align*}
Set $V_n = \bigcap_{k\geq n} U_k$ and $W_n = \bigcup_{k\geq n} U_k$ for all $n\in\N$.
The sequence $(V_n)_{n\in\N}$ is non-decreasing and $(W_n)_{n\in\N}$ is non-increasing.
We have the double inclusion, $V_n \subseteq U_n \subseteq W_n$ for all $n\in\N$, which leads to $U_n\setminus A \subseteq W_n\setminus A$ and $A\setminus U_n \subseteq A\setminus V_n$ for all $n\in\N$. By $\sigma$-additivity of the measure $m$, the quantities $m(A\setminus V_n)$ and $m(W_n\setminus A)$ tend to $0$ as $n$ goes to $\infty$. Therefore, $E[B_{U_n}-B_A]^2 \rightarrow 0$ as $n\rightarrow\infty$ and consequently $B_{U_n}$ converges to $B_A$ in probability.

\item {\em Set-indexed homogeneous Poisson process:}
A process $N=\{N_U;\;U\in\mathcal{A}\}$ with independent increments and such that for all $U\in\mathcal{A}$, $N_U$ has a Poisson distribution with parameter $c.m(U)$ (where $c>0$).\\
Following \cite{kingman,Ivanoff}, a Poisson process $N$ is equivalently defined by the representation
 $N_U=\sum_j\mathbbm{1}_{\{\tau_j\in U\}}$ for $U\in\mathcal{A}$ where the sequence $(\tau_j)_{j}$ of random points of
  $\mathcal{T}$ is measurable, and $(\tau_j)_{j}$ are uniformly distributed.

\begin{proposition}
The set-indexed homogeneous Poisson process is a set-indexed L\'evy process.
\end{proposition}

\begin{proof}
For any $C\in\mathcal{C}$, the increment $\Delta N_C$ can be written
\begin{equation}\label{eq:incPoisson}
\Delta N_C = \sum_j \mathbbm{1}_{\{ \tau_j \in C \}}.
\end{equation}
Consequently, the general definition of a Poisson process (see
\cite{kingman}) shows that $\{ \Delta N_C;\; C\in\mathcal{C} \}$
is a Poisson process indexed by the collection $\mathcal{C}$ and
therefore, the conditions (2') and (3') of Proposition
\ref{defLevyequiv} are satisfied.

\noindent
For the stochastic continuity, consider a sequence $(U_n)_{n\in\N}$ in $\mathcal{A}$ as in (\ref{eq:cond4}). For any $0<\epsilon<1$,
\begin{align*}
P(|N_{U_n}-N_A| > \epsilon)
&\leq P(| N_{U_n}-N_{U_n\cap A} | > \epsilon) + P(| N_{A}-N_{U_n\cap A} | > \epsilon) \\
&\leq \underbrace{P(| \Delta N_{U_n\setminus A} | \neq 0)}_{1-e^{-c\ m(U_n\setminus A)}} + \underbrace{P(| \Delta N_{A\setminus U_n} | \neq 0)}_{1-e^{-c\ m(A\setminus U_n)}}.
\end{align*}
As in the Brownian case, we conclude that $P(|N_{U_n}-N_A| > \epsilon) \rightarrow 0$ as $n\rightarrow\infty$.
\end{proof}

\item {\em Set-indexed compound Poisson process:}
A process $X=\{X_U;\;U\in\mathcal{A}\}$ is called a set-indexed compound Poisson process if it admits a representation
$$\forall U\in\mathcal{A};\quad X_U=\sum_{j}X_j \mathbbm{1}_{\{\tau_j \in U\}}$$
where $(X_n)_{n\in\mathbf{N}}$ is a sequence of i.i.d. real random variables and
$N=\{N_U, U\in\mathcal{A}\}$ is a set-indexed Poisson process of mean measure $\mu=c.m$ ($c>0$), defined by
$N_U=\sum_{j}\mathbbm{1}_{\{\tau_j \in U\}}$ for all $U\in\mathcal{A}$, independent of the sequence $(X_n)_{n\in\mathbf{N}}$.

Notice that the set-indexed compound Poisson process is an extension of the real-parameter compound Poisson process (take $\mathcal{A}=\left\{[0,t];\;t\in\mathbf{R}_+\right\}$).

\begin{proposition}
If $X=\{X_U;\;U\in\mathcal{A}\}$ is a set-indexed compound Poisson process,
then $X$ is a set-indexed L\'evy process and for all $U\in\mathcal{A}$ the distribution $\mu_U$ of $X_U$ satisfies:
\begin{align}\label{eq:compound_charact}
\forall z\in\mathbf{R};\quad\hat{\mu}_U(z)=\exp\left[c\ m(U)(\hat{\sigma}(z)-1)\right]
\end{align}
for some $c>0$ and some probability distribution $\sigma$.
\end{proposition}

\begin{proof}
For any $C\in\mathcal{C}$, the increment $\Delta X_C$ can be written
\begin{equation}\label{eq:incCompound}
\Delta X_C = \sum_j X_j \mathbbm{1}_{\{ \tau_j \in C \}}.
\end{equation}
We compute the characteristic function of $\Delta X_C$: For all $\lambda\in\mathbf{R}$, we have
\begin{align*}
E[e^{i\lambda \Delta X_C}] &= E\left( E[e^{i\lambda \Delta X_C} \mid \sigma(\tau_1,\tau_2,\dots)] \right)\\
&= E\left[ \varphi(\lambda)^{\Delta N_C} \right]
\end{align*}
where $\varphi$ denotes the characteristic function of $X_0$. We
used the fact that, conditionally to the $\tau_j$'s, $\Delta X_C$
is the sum of $\Delta N_C$ i.i.d. random variables. Then
\begin{align*}
E\left[ \varphi(\lambda)^{\Delta N_C} \right]
&= \sum_j E[ \varphi(\lambda)^{\Delta N_C} \mid \Delta N_C = j]. P\{\Delta N_C=j\} \\
&= \sum_j \varphi(\lambda)^j . P\{\Delta N_C=j\} \\
&= \sum_j \left( \varphi(\lambda) \right)^j  e^{-\mu(C)}\frac{\mu(C)^j}{j!}.
\end{align*}
We conclude the computation
\begin{align*}
E[e^{i\lambda \Delta X_C}] = e^{-\mu(C)} \sum_j \frac{(\varphi(\lambda).\mu(C))^j}{j!}
= e^{-\mu(C)}\ e^{\varphi(\lambda)\mu(C)}
= e^{\mu(C)[\varphi(\lambda) - 1]}.
\end{align*}
This relation proves the stationarity condition (3') of Proposition \ref{defLevyequiv}.

To prove the independence condition (2'), let us consider two
subsets $C_1, C_2\in\mathcal{C}$ such that $C_1\cap
C_2=\emptyset$. We remark that the same computation leads to
\begin{align*}
E[e^{i\lambda (\Delta X_{C_1}+\Delta C_{C_2})}]
= e^{(\mu(C_1)+\mu(C_2))[\varphi(\lambda) - 1]}
= E[e^{i\lambda \Delta X_{C_1}}] . E[e^{i\lambda \Delta X_{C_1}}],
\end{align*}
which proves the independence of $\Delta X_{C_1}$ and $\Delta
X_{C_2}$.

\noindent For the stochastic continuity, consider a sequence
$(U_n)_{n\in\N}$ in $\mathcal{A}$ as in (\ref{eq:cond4}). From the
structure (\ref{eq:compound_charact}) of the characteristic
function of $X_{U_n}$, we deduce that $\hat{\mu}_{U_n}$ converges
to $1$ as $n\rightarrow\infty$. Then $X_{U_n}$ converges in law to
$0$ and thus in probability.
\end{proof}

\end{itemize}
\end{example}

All the previous examples generate a vector space which is included in the set of L\'evy processes. In the next section, we will prove that the closure of this vector space in some sense constitutes exactly the class of L\'evy processes.

\section{Infinitely divisible laws}\label{sectdiv}

In this section, we show that the law of any set-indexed L\'evy process $\{X_U;\;U\in\mathcal{A}\}$ is characterized by a single infinitely divisible distribution.
The first step is to prove that its marginal distributions, $P_{X_U}$ where $U\in\mathcal{A}$, are infinitely divisible.

This fact is improved in the proof of Theorem \ref{thRepCanon} for the marginal distributions of the increment process $\Delta X$.
There exists an infinitely divisible probability measure $\mu$ such that $P_{\Delta X_C}=\mu^{m(C)}$ for all $C\in\mathcal{C}$.
As a consequence, following the terminology of \cite{RR89}, the $\mathcal{C}$-process $\Delta X$ is an infinitely divisible independently scattered random measure on the $\delta$-ring $\mathcal{C}$.
Then, Proposition 2.1 of \cite{RR89} provides a characterization of the ID random measure $\Delta X$ in terms of its L\'evy-Khintchine representation
\begin{equation}\label{eq:IDISRM-LK}
E\left[e^{iz \Delta X_C}\right] = \exp\left\{ iz \nu_0(C) -\frac{1}{2}z^2 \nu_1(C)
+\int_{\R} \left[ e^{izx} - 1 - iz \tau(x) \right]\ F_C(dx) \right\},
\end{equation}
where $\tau(x)=x\ \mathbbm{1}_{|x|\leq 1} + (x/|x|) \mathbbm{1}_{|x|>1}$, $\nu_0$ is a signed measure, $\nu_1$ is a measure, $F_C$ is a L\'evy measure for all $C\in\mathcal{C}$, and for any $B\in\mathcal{B}(\R)$ such that $0\notin \overline{B}$, $C\mapsto F_C(B)$ is a measure.
Moreover, a form of stochastic continuity holds. The measure $\lambda$ defined by
\begin{equation*}
\forall C\in\mathcal{C},\quad
\lambda(C) = |\nu_0|(C) + \nu_1(C) + \int_{\R} (1\wedge x^2)\ F_C(dx),
\end{equation*}
satisfies: For any sequence $(C_n)_{n\in\N}$ in $\mathcal{C}$ such that $\lambda(C_n)\rightarrow 0$ as $n\rightarrow\infty$, $\Delta X_{C_n}$ converges to $0$ in probability.

The main result of this section, Theorem \ref{thRepCanon}, takes advantage of the combination of the joint increment independence and increment stationarity properties. An explicit expression (\ref{eq:deltaX-fdd}) for the finite-dimensional distributions of the $\mathcal{C}$-indexed process $\Delta X$ is obtained. This result is essential for the next characterizations of the set-indexed L\'evy processes : among the class of Markov processes, the L\'evy-It\^o and semimartingale representations (see Sections 6 and 7).
As a consequence of Theorem \ref{thRepCanon}, a L\'evy-Khintchine representation (\ref{eq:LK}) is derived for the law of this process. The previous measures $\nu_0$, $\nu_1$ and $C\mapsto F_C(B)$ are proportional to $m(C)$.

\

The following result will be necessary for infinitely divisibility of marginal laws of a set-indexed L\'evy process.

\begin{proposition}\label{HYP}
If $m$ is a Radon measure, then for any $U\in\mathcal{A}$ and for all integer $n$, there exists a family $(C_i)_{1\leq i\leq n}$ in $\mathcal{C}_0$ such that
\begin{enumerate}[(i)]
\item $\forall i\ne j$, $C_i\cap C_j=\emptyset$;

\item $\forall i,j\in\{1,\dots,n\}$, $m(C_i)=m(C_j)$;

\item and $$U=\bigcup_{1\leq i\leq n}C_i.$$
\end{enumerate}
The family $(C_i)_{1\leq i\leq n}$ is called a {\em $m$-partition of size $n$} of $U$.
\end{proposition}

\begin{proof}
For any $U\in\mathcal{A}$, Lemma 3.3 in \cite{flow} (or Lemma
5.1.6 in \cite{Ivanoff}) implies existence of an elementary flow
$f:\mathbf{R}_+\rightarrow\mathcal{A}$ such that $f(0)=\emptyset$
and $f(1)=U$. By continuity of $t\mapsto m[f(t)]$ from $[0,1]$ to
$[0,m(U)]$, there exists $t_1\in\mathbf{R}_+$ such that
$U_1=f(t_1)=\frac{m(U)}{n}$.

In the same way, the continuity of $t\mapsto m[f(t)]$ implies the existence of $t_2,t_3,\dots,t_{n-1}$ such that $0<t_1<t_2<\dots<t_{n-1}<t_n=1$ and
\begin{equation*}
\forall i=2,\dots,n;\quad
m[f(t_{i})]-m[f(t_{i-1})] = \frac{m(U)}{n}.
\end{equation*}
Setting $U_i=f(t_i)$ for all $i=2,\dots,n$, we get
$U_1\subseteq U_2\subseteq\dots\subseteq U_{n-1}\subseteq U_n=U$ and
$$m(U_2\setminus U_1)=m(U_3\setminus U_2)=\dots=m(U\setminus U_{n-1})=\frac{m(U)}{n}.$$
(It suffices to remark that $m(U_i\setminus
U_{i-1})=m(U_i)-m(U_{i-1})$ for all $i=2,\dots,n$.)

The family of $C_i=U_i\setminus U_{i-1}$ for $2\leq i\leq n$ and $C_1=U_1$ satisfies all the conclusions of the proposition.
\end{proof}

\begin{corollary}\label{corinfdiv}
If $X=\left\{X_U;\;U\in\mathcal{A}\right\}$ is a set-indexed L\'evy process on $(\mathcal{T},\mathcal{A},m)$, then for all $U\in\mathcal{A}$, the distribution of $X_U$ is infinitely divisible.
\end{corollary}

\begin{proof}
For any integer $n$, let us consider a $m$-partition of $U=\bigcup_{1\leq i\leq n}C_i$, where for all $i\ne j$, $C_i\cap C_j=\emptyset$ and $m(C_i)=m(C_j)=m(U)/n$.
The definition of the increment process gives
\begin{equation}\label{decompinc}
X_U = \sum_{i=1}^n \Delta X_{C_i}.
\end{equation}
By definition of the set-indexed L\'evy process, the $\Delta X_{C_i}$ are i.i.d. and let us denote by $\mu_n$ their distribution. By equation (\ref{decompinc}), the distribution $\mu$ of $X_U$ can be written as
\begin{equation*}
\mu = \mu_n * \dots * \mu_n = (\mu_n)^n
\end{equation*}
and therefore, $\mu$ is infinitely divisible.
\end{proof}

The infinitely divisible property of the marginal laws is improved in the following theorem. We prove that the increments of a set-indexed L\'evy process constitute an ID independently scattered random measure, which which does not directly follow from Definition \ref{deflevy} but precisely relies on $m$-stationarity on $\mathcal{C}_0$ and continuity of flows.

\begin{theorem}[Canonical Representation]\label{thRepCanon}
The three following statements hold:
\begin{enumerate}[(i)] 
\item Let $X=\{X_U\; U\in\mathcal{A}\}$ be a set-indexed L\'evy process and $U_0\in\mathcal{A}$ such that $m(U_0)>0$.
For all $U\in\mathcal{A}$, the distribution of $X_U$ is equal to $\mu^{m(U)}$ where $\mu=(P_{X_{U_0}})^{1/m(U_0)}$. 
Moreover, the law of the L\'evy process $X$ is completely determined by the law of $X_{U_0}$.
\item A set-indexed process $\{X_U;\;U\in\mathcal{A}\}$ is a set-indexed L\'evy process if and only if its increment process $\{\Delta X_C;\;C\in\mathcal{C}\}$ is an infinitely divisible independently scattered random measure which satisfies the $m$-stationarity property, i.e. such that the distributions of $\Delta X_{C_1}$ and $\Delta X_{C_2}$ are equal if $m(C_1)=m(C_2)$.

\item For any infinitely divisible probability measure $\mu$ on $(\mathbf{R},\mathcal{B})$, there exists a set-indexed L\'evy process $X$ such that
\begin{equation}\label{infeq}
\forall U\in\mathcal{A};\quad
P_{X_U} = \mu^{m(U)}.
\end{equation}
\end{enumerate}

\noindent In the three statements, for all $C_0,\dots,C_n\in\mathcal{C}$ and all Borel sets $B_1,\dots,B_n$, the finite dimensional distributions of the increment process $\Delta X$ are given by
\begin{align}\label{eq:deltaX-fdd}
P&\left( \Delta X_{C_0}\in B_0, \dots,  \Delta X_{C_n}\in B_n \right) \nonumber\\
&=
\int \mu^{m(\cap^{(n)})}(d\xi^{(n)}) \prod_{j_1<\dots<j_{n-1}} \mu^{m(\cap^{(n-1)}_{j_1,\dots,j_{n-1}}\setminus \cap^{(n)})}(d\xi^{(n-1)}_{j_1,\dots,j_{n-1}})
\dots \prod_{i<j}\mu^{m(\cap^{(2)}_{i,j}\setminus \cap^{(3)})}(d\xi^{(2)}_{i,j}) \nonumber\\
&\quad\times\prod_i \mu^{m(C_i\setminus \cap^{(2)})}(d\xi^{(1)}_{i})
 \mathbbm{1}_{B_i}\Bigg\{\xi^{(n)}+\sum_{k=1}^{n-1}
 \sum_{ \genfrac{}{}{0pt}{}{j_1<\dots<j_k}{i\in\{j_1,\dots,j_k\}} }
 \xi^{(k)}_{j_1,\dots,j_k} \Bigg\},
\end{align}
using the notation
\begin{align*}
\cap^{(2)}_{i,j} = C_i\cap C_j ;\quad
&\cap^{(3)}_{i,j,k} = C_i\cap C_j\cap C_k ;\quad
\dots \quad
\cap^{(n)} = C_1\cap\dots\cap C_n ;\\
\cap^{(2)} = \bigcup_{i<j}\cap^{(2)}_{i,j} ;\quad
&\cap^{(3)} = \bigcup_{i<j<k} \cap^{(3)}_{i,j,k} ;\quad
\dots
\end{align*}

\end{theorem}

\

\begin{proof}
$(i)$
Let $\mu=P_{X_{U_0}}$. As in Corollary \ref{corinfdiv}, for any integer $n$, we consider a $m$-partition of $U_0=\bigcup_{1\leq i\leq n}C_i$, where $C_i\cap C_j=\emptyset$ for all $i\ne j$ and $m(C_i)=m(C_j)=m(U_0)/n$. 
We have $\mu=\left(P_{\Delta X_{C_1}}\right)^n$ and then $P_{\Delta X_{C_1}}=\mu^{1/n}=\mu^{m(C_1)/m(U_0)}$. Then the increment stationarity property implies that $P_{\Delta X_C}=\mu^{m(C)/m(U_0)}$ for any $C\in\mathcal{C}_0$ with $m(U_0)/m(C)\in\mathbf{N}$.

\noindent
For any element $U\in\mathcal{A}$ of measure $m(U)=m(U_0)\ p/n$ with $n,p\in\mathbf{N}^*$, we can use the same way to decompose $P_{X_U}=\left(\mu^{1/n}\right)^p= \mu^{m(U)/m(U_0)}$.

\noindent
More generally, for all element $U\in\mathcal{A}\setminus\{\mathcal{\emptyset,T}\}$ with $m(U)/m(U_0)\in\R_+\setminus\Q_+$, we can consider a set $V\in\mathcal{A}$ such that $U\subsetneq V$ and $m(U)<m(V)$. There exists an elementary flow $f:[0,1]\rightarrow\mathcal{A}$ such that $f(0)=\emptyset'=\bigcap_{W\in\mathcal{A}}W$, $f(1/2)=U$ and $f(1)=V$ (since $U\neq\emptyset$, we have $\emptyset'\subseteq U$). By continuity of $t\mapsto m[f(t)]$, we can construct a sequence $(t_n)_{n\in\mathbf{N}}$ in $[1/2,1]$ decreasing to $1/2$ such that $m[f(t_n)]/m(U_0)\in\mathbf{Q}_+^*$ for all $n$. For all $n$, we have $P_{X_{f(t_n)}}= \mu^{m[f(t_n)]/m(U_0)}$ and then stochastic continuity implies $P_{X_U}=\mu^{m(U)/m(U_0)}$, which proves that $\mu$ determines all marginal laws of $X$.
Since $X_{\emptyset}=0$, the expression $P_{X_U}=\mu^{m(U)/m(U_0)}$ also holds for $U=\emptyset$.
This proves the first part of $(i)$.

\

\noindent
In order to prove $(ii)$, the previous result is improved showing that the distribution of any increment $\Delta X_C$ with $C\in\mathcal{C}$ is also determined by $\mu$ as $P_{\Delta X_C}=\mu^{m(C)/m(U_0)}$.
We first consider the case of $C=U\setminus V\in\mathcal{C}_0$, where $U,V\in\mathcal{A}$. We have $X_U = \Delta X_{U\setminus V} + X_{U\cap V}$ and then, using
\begin{equation*}
\forall s,t\in\mathbf{R}_+;\quad\mu^{s+t} = \mu^s * \mu^t
\Leftrightarrow \hat{\mu}^{s+t} = \hat{\mu}^s \hat{\mu}^t
\end{equation*}
where $\hat{\mu}$ denotes the characteristic function of the measure $\mu$, and independence of $\Delta X_{U\setminus V}$ and $X_{U\cap V}$, we get
\begin{align*}
\widehat{P_{X_U}} = \widehat{P_{\Delta X_{U\setminus V}}} \widehat{P_{X_{U\cap V}}}
\Rightarrow
\widehat{P_{\Delta X_{U\setminus V}}} &= \hat{\mu}^{(m(U)-m(U\cap V))/m(U_0)}\\
&= \hat{\mu}^{m(U\setminus V)/m(U_0)}.
\end{align*}
By definition of $\mu^t$, it leads to $P_{\Delta X_{U\setminus V}}=\mu^{m(U\setminus V)/m(U_0)}$.

\noindent
In the same way, for all $C=U\setminus\bigcup_{1\leq i\leq n}U_i \in\mathcal{C}$ where $U,U_1,\dots,U_n\in\mathcal{A}$, we write
\begin{equation*}
\Delta X_{U\setminus\bigcup_{1\leq i\leq n-1}U_i} =
\Delta X_{U\setminus\bigcup_{1\leq i\leq n}U_i} + \Delta X_{U\cap U_n\setminus\bigcup_{1\leq i\leq n-1}U_i}.
\end{equation*}
Using the independence of $\Delta X_{U\setminus\bigcup_{1\leq i\leq n}U_i}$ and $\Delta X_{U\cap U_n\setminus\bigcup_{1\leq i\leq n-1}U_i}$, we can deduce
\begin{equation}\label{eqPdeltaXC}
\forall C=U\setminus\bigcup_{1\leq i\leq n}U_i\in\mathcal{C};\quad
P_{\Delta X_C} = \mu^{m(C)/m(U_0)}
\end{equation}
by induction on $n$. 
Decomposing elements of $\mathcal{C}(u)$ as disjoint unions of elements in $\mathcal{C}$, (\ref{eqPdeltaXC}) can be extended in
\begin{equation}\label{eqPdeltaXCu}
\forall C\in\mathcal{C}(u);\quad
P_{\Delta X_C} = \mu^{m(C)/m(U_0)}.
\end{equation}

Since $\Delta X$ is known to be an independently scattered measure (cf. Section \ref{sec:def}), expression (\ref{eqPdeltaXC}) implies that $\Delta X$ is an infinitely divisible ISRM, such that the distributions of $\Delta X_{C_1}$ and $\Delta X_{C_2}$ are equal if $m(C_1)=m(C_2)$. This proves the direct part of $(ii)$.

For the converse part of $(ii)$, assume that the increment process $\Delta X$ of an $\mathcal{A}$-indexed process is an infinitely divisible ISRM which satisfies the $m$-stationarity property. Then the characteristic function of $\Delta X_C$ is given by (\ref{eq:IDISRM-LK}) (see Proposition~2.1 of \cite{RR89}).
For any $C\in\mathcal{C}$ and $B\in\mathcal{B}(\R)$ with $0\notin\overline{B}$, the quantities $\nu_0(C)$, $\nu_1(C)$ and $F_C(B)$ only depend on $m(C)$.
The finite additivity in $C$ implies that they are proportional to $m(C)$.

\noindent For instance, writing $\nu_0(C) = \Phi(m(C))$, since $\nu_0$ is a signed measure, we have for all disjoints subsets $C_1$ and $C_2$ in $\mathcal{C}$ such that $C_1\cup C_2\in\mathcal{C}$,
\begin{align*}
\Phi(m(C_1)+m(C_2)) = \nu_0(C_1\cup C_2) = \nu_0(C_1) + \nu_0(C_2)
= \Phi(m(C_1)) + \Phi(m(C_2)).
\end{align*}
This additivity of $\Phi$ implies that $\Phi(m(C)) = a\ m(C)$ with $a\in\R$.
The same argument applied to the measures $\nu_1$ and $C\mapsto F_C(B)$ leads to
\begin{equation*}
E\left[e^{iz \Delta X_C}\right] = \exp\left( m(U) \left\{ i a z -\frac{1}{2}z^2 b
+\int_{\R} \left[ e^{izx} - 1 - iz \tau(x) \right]\ \nu(dx) \right\}\right),
\end{equation*}
where $a,b\in\R$, $b>0$ and $\nu$ is a measure.
Considering the infinitely probability measure $\mu$ with characteristic function
\begin{equation*}
\int e^{iz x}\mu(dx) = \exp\left\{ i a z -\frac{1}{2}z^2 b
+\int_{\R} \left[ e^{izx} - 1 - iz \tau(x) \right]\ \nu(dx) \right\},
\end{equation*}
we deduce $P_{\Delta X_C} = \mu^{m(C)}$ for all $C\in\mathcal{C}$.
This relation proves the $m$-stationarity of $\mathcal{C}_0$-increments of $X$, which completes the independence of increments which follows the fact that $\Delta X$ is an ISRM.

It remains to prove the stochastic continuity of $X$.
Let $(U_n)_{n\in\mathbf{N}}$ be a sequence in $\mathcal{A}$ such that
\begin{equation*}
\overline{\bigcup_{n} \bigcap_{k\geq n} U_k} = \bigcap_n \overline{\bigcup_{k\geq n} U_k} = A\in\mathcal{A}.
\end{equation*}
For all $n\in\mathbf{N}$, we have
\begin{align*}
X_{U_n} - X_A &= X_{U_n} - X_{U_n\cap A} + X_{U_n\cap A} - X_A \\
&= \Delta X_{U_n\setminus A} - \Delta X_{A\setminus U_n}.
\end{align*}
Since $(U_n\setminus A) \cap (A\setminus U_n) = \emptyset$,
$\Delta X_{U_n\setminus A}$ and $\Delta X_{A\setminus U_n}$ are
independent and the distribution of $\Delta X_{U_n\setminus A} -
\Delta X_{A\setminus U_n}$ is the convolution product of the laws
of $\Delta X_{U_n\setminus A}$ and $-\Delta X_{A\setminus U_n}$.
Then
\begin{align*}
P(|X_{U_n} - X_A| > \epsilon) &= \int\int \mathbbm{1}(|x-y|>\epsilon)\  \mu^{m(U_n\setminus A)}(dx)\ \mu^{m(A\setminus U_n)}(dy).
\end{align*}
By definition of $(U_n)_n$, we have $\lim_{n\rightarrow\infty}
m(U_n\setminus A)=0$ and $\lim_{n\rightarrow\infty} m(A\setminus
U_n)=0$. Using $\mu^t\rightarrow\delta_0$ as $t\rightarrow 0$, and
the boundedness of $\mathbbm{1}(|x-y|>\epsilon)$, we get
\begin{align*}
P(|X_{U_n} - X_A| > \epsilon) \rightarrow 0\quad\textrm{as }n\rightarrow\infty.
\end{align*}
This achieves the proof of $(ii)$.

\

To complete the proof of $(i)$, it remains to prove that $\mu=P_{X_{U_0}}$ completely determines the law of the process $X$.
Without any loss of generality, we assume that $m(U_0)=1$ (if not, consider $m(\bullet)/m(U_0)$ instead of $m(\bullet)$).

\noindent
For all $C_0$ and $C_1$ in $\mathcal{C}$, using additivity of $\Delta X$ we can decompose
\begin{align*}
\Delta X_{C_0} &= \Delta X_{C_0\setminus (C_0\cap C_1)} + \Delta X_{C_0\cap C_1} \\
\Delta X_{C_1} &= \Delta X_{C_1\setminus (C_0\cap C_1)} + \Delta X_{C_0\cap C_1}
\end{align*}
where $\Delta X_{C_0\cap C_1}$, $\Delta X_{C_0\setminus (C_0\cap C_1)}$ and $\Delta X_{C_1\setminus (C_0\cap C_1)}$ are pairwise independent. Then, conditionally to $ \Delta X_{C_0\cap C_1}$, the random variables $\Delta X_{C_0}$ and $\Delta X_{C_1}$ are independent. \\
We use this fact to compute for all Borel sets $B_0$ and $B_1$
\begin{align*}
P&\left( \Delta X_{C_0}\in B_0, \Delta X_{C_1}\in B_1 \right) \\
&=
\int P\left( \Delta X_{C_0}\in B_0, \Delta X_{C_1}\in B_1 \mid \Delta X_{C_0\cap C_1} \right) .P_{\Delta X_{C_0\cap C_1}}(d\xi) \\
&= \int P\left(\Delta X_{C_0}\in B_0 \mid \Delta X_{C_0\cap C_1} \right)
P\left(\Delta X_{C_1}\in B_1 \mid \Delta X_{C_0\cap C_1} \right) .P_{\Delta X_{C_0\cap C_1}}(d\xi) \\
&= \int P\left(\Delta X_{C_0\setminus (C_0\cap C_1)} + \xi\in B_0 \right)
P\left(\Delta X_{C_1\setminus (C_0\cap C_1)} + \xi\in B_1 \right) .P_{\Delta X_{C_0\cap C_1}}(d\xi),
\end{align*}
using independence of $\Delta X_{C_0\cap C_1}$ with $\Delta X_{C_0\setminus (C_0\cap C_1)}$ and $\Delta X_{C_1\setminus (C_0\cap C_1)}$.\\
Then we get the expression for the distribution of $(\Delta
X_{C_0},\Delta X_{C_1})$
\begin{align*}
P&\left( \Delta X_{C_0}\in B_0, \Delta X_{C_1}\in B_1 \right) \\
&=\int P_{\Delta X_{C_0\cap C_1}}(d\xi)\ \mathbbm{1}_{B_0}(y_0+\xi)\ P_{\Delta X_{C_0\setminus (C_0\cap C_1)}}(dy_0)\ \mathbbm{1}_{B_1}(y_1+\xi)\ P_{\Delta X_{C_1\setminus (C_0\cap C_1)}}(dy_1) \\
&=\int \mu^{m(C_0\cap C_1)}(d\xi)\ \mathbbm{1}_{B_0}(y_0+\xi)\ \mu^{m(C_0\setminus (C_0\cap C_1))}(dy_0)\ \mathbbm{1}_{B_1}(y_1+\xi)\ \mu^{m(C_1\setminus (C_0\cap C_1))}(dy_1),
\end{align*}
using expression (\ref{eqPdeltaXC}).

\noindent
More generally, for all $C_0,\dots,C_n\in\mathcal{C}$, we introduce the notation
\begin{align*}
\cap^{(2)}_{i,j} = C_i\cap C_j ;\quad
&\cap^{(3)}_{i,j,k} = C_i\cap C_j\cap C_k ;\quad
\dots \quad
\cap^{(n)} = C_1\cap\dots\cap C_n ;\\
\cap^{(2)} = \bigcup_{i<j}\cap^{(2)}_{i,j} ;\quad
&\cap^{(3)} = \bigcup_{i<j<k} \cap^{(3)}_{i,j,k} ;\quad
\dots
\end{align*}
Each random variable $\Delta X_{C_i}$ can be decomposed in
\begin{align*}
\Delta X_{C_i} = \Delta X_{\cap^{(n)}}
+ \sum_{ \genfrac{}{}{0pt}{}{j_1<\dots<j_{n-1}}{i\in\{j_1,\dots,j_{n-1}\}} }\Delta X_{\cap^{(n-1)}_{j_1,\dots,j_{n-1}}\setminus \cap^{(n)}}
+ \dots
+ \sum_{j\ne i}\Delta X_{\cap^{(2)}_{i,j}\setminus \cap^{(3)}}
+ \Delta X_{C_i\setminus \cap^{(2)}}.
\end{align*}
As in the case $n=2$, we get expression (\ref{eq:deltaX-fdd}) for $P\left( \Delta X_{C_0}\in B_0, \dots,  \Delta X_{C_n}\in B_n \right)$.
This expression shows that the law of the process $X$ is completely determined by $\mu$, i.e. by the law of $X_{U_0}$. Assertion $(i)$ is now proved.
\vspace{10pt}

For the proof of $(iii)$, let $\mu$ be an infinitely divisible measure. We aim at constructing a L\'evy process $X$ such that condition (\ref{infeq}) holds. 
For the sake of simplicity, we will construct the increment process $\Delta X$ indexed by $\mathcal{C}$ rather than $X$. We consider the canonical space $\Omega=\mathbf{R}^{\mathcal{C}}$ where any $\mathcal{C}$-indexed process $Y$ will be defined by $Y_C(\omega)=\omega(C)$ ($C\in\mathcal{C}$). 
As usual, $\Omega$ is endowed with the $\sigma$-field $\mathcal{F}$ generated by the cylinders
\begin{equation*}
\Lambda=\left\{ \omega\in\Omega: Y_{C_1}(\omega)\in B_1, \dots, Y_{C_n}(\omega)\in B_n \right\},
\end{equation*}
where $C_1,\dots,C_n\in\mathcal{C}$ and $B_1,\dots,B_n\in\mathcal{B}(\mathbf{R})$.

\noindent
As in the classical context of real-parameter L\'evy processes (see \cite{sato}), for all $t\in\mathbf{R}_+$, $\mu^t$ is defined and satisfies
\begin{eqnarray}\label{puissnu}
\left. \begin{array}{c}
\forall s,t\in\mathbf{R}_+;\quad
\mu^s * \mu^t = \mu^{s+t} \\
\mu^0 = \delta_0 \\
\mu^t \rightarrow \delta_0 \quad\textrm{as }t\rightarrow 0.
\end{array} \right\}
\end{eqnarray}

\noindent
For any $n\geq 0$ and any $C_0, C_1,\dots ,C_n$ in $\mathcal{C}$, we define for all Borel sets $B_0,B_1,\dots,B_n$,
\begin{align}\label{lawY}
\lambda&_{C_0,\dots,C_n}(B_0\times\dots\times B_n)\nonumber\\
&=\int \mu^{m(\cap^{(n)})}(d\xi^{(n)})
\prod_{j_1<\dots<j_{n-1}} \mu^{m(\cap^{(n-1)}_{j_1,\dots,j_{n-1}}\setminus \cap^{(n)})}(d\xi^{(n-1)}_{j_1,\dots,j_{n-1}})
\dots \prod_{i<j}\mu^{m(\cap^{(2)}_{i,j}\setminus \cap^{(3)})}(d\xi^{(2)}_{i,j}) \nonumber\\
&\quad\times\prod_i \mu^{m(C_i\setminus \cap^{(2)})}(d\xi^{(1)}_{i})
 \mathbbm{1}_{B_i}\Bigg\{\xi^{(n)}+\sum_{k=1}^{n-1}
 \sum_{ \genfrac{}{}{0pt}{}{j_1<\dots<j_k}{i\in\{j_1,\dots,j_k\}} }
 \xi^{(k)}_{j_1,\dots,j_k} \Bigg\},
\end{align}
using the notation of the direct part of the proof.

By definition of the product $\sigma$-field $\mathcal{B}(\mathbf{R}^{n+1})$, the additive function $\lambda_{C_0,\dots,C_n}$ can be extended to a measure. 
Using (\ref{puissnu}), the family of measures $(\lambda_{C_0,\dots,C_n})_{n,C_0,\dots,C_n}$ satisfies the usual consistency conditions. Following the general Kolmogorov extension Theorem (see \cite{kallenberg}, theorem 6.16), we get a probability measure $P$ such that the canonical process $Y$ has the finite dimensional distributions $\lambda_{C_0,\dots,C_n}$. In particular, $Y_C$ has distribution $\mu^{m(C)}$.

The set-indexed process is clearly finitely additive, in the sense that for all $C_1,C_2\in\mathcal{C}$ such that $C_1\cap C_2=\emptyset$ and $C_1\cup C_2\in\mathcal{C}$, we have $Y_{C_1\cup C_2}=Y_{C_1} + Y_{C_2}$ almost surely.
Then, if we define the $\mathcal{A}$-indexed process $X=\{X_U = Y_U;\; U\in\mathcal{A} \}$, the process $Y$ is exactly the increment process $\Delta X$ of $X$. Therefore, the distribution of $\Delta X_C$ is $\mu^{m(C)}$.
\vspace{5pt}

Let us show that $X$ is a set-indexed L\'evy process. From (\ref{lawY}), if we consider pairwise disjoint sets $C_1,\dots,C_n$ in $\mathcal{C}$, we get
\begin{align*}
P\left( \Delta X_{C_0}\in B_0, \dots,  \Delta X_{C_n}\in B_n \right)
&= \int \prod_i \mu^{m(C_i)}(d\xi^{(1)}_{i})\mathbbm{1}_{B_i}( \xi^{(1)}_{i} ) \\
&= \prod_i \int\mu^{m(C_i)}(d\xi^{(1)}_{i})\mathbbm{1}_{B_i}( \xi^{(1)}_{i} ) \\
&= \prod_i  P\left( \Delta X_{C_i}\in B_i \right),
\end{align*}
which proves the independence of $\Delta X_{C_0},\dots,\Delta X_{C_n}$. Then, since the distribution of $\Delta X_C$ only depends on $m(C)$, Lemma \ref{lemequivstat} implies the $m$-stationarity of the $\mathcal{C}_0$-increments of $X$.

The proof of the stochastic continuity of $X$ is identical as the proof of $(ii)$, since it only uses the relation $P_{\Delta X_C} = \mu^{m(C)}$ and independence of increments.
\end{proof}

\begin{remark}
In Theorem \ref{thRepCanon}, after having proved the equivalence between $X$ is a set-indexed L\'evy process and $\Delta X$ is an ID independently scattered random measure (with control measure proportional to $m$, assertion $(ii)$), assertion $(iii)$ is already proved in \cite{RR89} (Proposition 2.1).

However, the goal is to obtain the precise expression of the finite dimensional distributions of $\Delta X$, which opens the door to a characterisation of SI L\'evy processes in terms of Markov transition systems and also martingale properties.
\end{remark}

Another formulation of the canonical representation theorem is that the law of a set-indexed L\'evy process $X=\{X_U;\;U\in\mathcal{A}\}$ is completely determined by an infinitely divisible probability measure $\mu$, and that
\begin{equation*}
\forall U\in\mathcal{A};\quad
P_{X_U} = \mu^{m(U)}.
\end{equation*}
Thus, the L\'evy-Khintchine formula implies that the law of $X$ is characterized by a unique triplet $(\sigma,\gamma,\nu)$, where $\sigma\geq 0$, $\gamma\in\mathbf{R}$ and $\nu$ is a measure such that $\nu(\{ 0\}) = 0$ and $\int_{\mathbf{R}} \left[ |x|^2\wedge 1  \right] \ \nu(dx) <+\infty$.
For any $U\in\mathcal{A}$, the law of $X_U$ has the characteristic function $E[e^{i z X_U}] = \exp\Psi_U(z)$, where
\begin{equation}\label{eq:LK}
\Psi_U(z) = m(U) \left\{ -\frac{1}{2}\sigma^2 z^2 + i \gamma z
+\int_{\mathbf{R}} \left[ e^{i z x} - 1 - i z x\ \mathbbm{1}_D(x) \right] \ \nu(dx) \right\}
\end{equation}
with $D=\{x: |x| \leq 1 \}$.

\

As mentioned in Section \ref{sec:frame}, the natural definition for increment stationarity of set-indexed processes only concerns the sub-class $\mathcal{C}_0$ of $\mathcal{C}$. It is not obvious that this property implies the {\it a priori} stronger $m$-stationarity of the $\mathcal{C}$-increments.
In the proof of Theorem \ref{thRepCanon}, we observe that the $m$-stationarity on $\mathcal{C}_0$ is sufficient to derive the finite dimensional distributions of $\{\Delta X_C;\;C\in\mathcal{C}\}$, and consequently, the stationarity property on the whole $\mathcal{C}$ holds for set-indexed L\'evy processes (thanks to independence of increments and infinite divisibility property). 

The expression $P_{\Delta X_C} = \mu^{m(C)/m(U_0)}$ (for all $C\in\mathcal{C}$) clearly implies condition $(3')$ of Proposition \ref{defLevyequiv}. And conversely, we have just proved that if $X$ is a set-indexed L\'evy process, then the distribution of $\Delta X_C$ only depends on $m(C)$. Therefore, we can state the equivalent definition for set-indexed L\'evy processes:

\begin{corollary}
A set-indexed process $X=\{X_U;\;U\in\mathcal{A}\}$ is a set-indexed L\'evy process if and only if the following four conditions hold:
\begin{enumerate}
\item $X_{\emptyset'}=0$ almost surely.

\item for all pairwise disjoint sets $C_1,\dots,C_n$ in $\mathcal{C}$, the random variables $\Delta X_{C_1},\dots,\Delta X_{C_n}$ are independent.

\item for all $C_1, C_2\in \mathcal{C}$, we have
\begin{equation*}
m(C_1)=m(C_2)
\Rightarrow
\Delta X_{C_1} \stackrel{(d)}{=}
\Delta X_{C_2}.
\end{equation*}

\item if $(U_n)_{n\in\N}$ is a sequence in $\mathcal{A}$ such that
\begin{equation*}
\overline{\bigcup_{n} \bigcap_{k\geq n} U_k} = \bigcap_n \overline{\bigcup_{k\geq n} U_k} = A \in\mathcal{A}
\end{equation*}
then
\begin{equation*}
\lim_{n\rightarrow\infty} P\left\{ |X_{U_n} - X_A| > \epsilon \right\} = 0.
\end{equation*}

\end{enumerate}
\end{corollary}

\section{Projection on flows}

The notion of flow is a key to reduce the proof of many theorems in the set-indexed theory and this notion was extensively studied in \cite{cime} and \cite{Ivanoff}. However, set-indexed processes should not be seen as a simple collection of real-parameter processes corresponding to projections on flows. Moreover, for the general indexing collection $\mathcal{A}$, we cannot expect to obtain a characterization of set-indexed L\'evy in terms of flows.
In particular, problems of existence of set-indexed processes, continuity in probability and increment independence cannot be addressed by their analogues on flows.

As we will show, projections of set-indexed L\'evy processes on flows generally are classical L\'evy processes, but the converse does not hold: The set-indexed L\'evy process has a very rich structure.
However, the notion of $m$-stationarity of $\mathcal{C}_0$-increments is well adapted to some classes of flows.

In this section, we define two types of flows, the elementary flows which take their values in the collection $\mathcal{A}$ and the simple flows which are finite unions of elementary flows and therefore taking their values in class $\mathcal{A}(u)$.

The main result shows the various relations between set-indexed processes and their projections on different flows.

\begin{definition}\label{flowdef}
An {\em elementary flow} is defined to be a continuous increasing function $f:[a,b]\subset\mathbf{R}_+\rightarrow\mathcal{A}$, i. e. such that
\begin{align*}
\forall s,t\in [a,b];\quad & s<t \Rightarrow f(s)\subseteq f(t)\\
\forall s\in [a,b);\quad & f(s)=\bigcap_{v>s}f(v)\\
\forall s\in (a,b);\quad & f(s)=\overline{\bigcup_{u<s}f(u)}.
\end{align*}

A {\em simple flow} is a continuous function $f:[a,b]\rightarrow\mathcal{A}(u)$ such that there exists a finite sequence $(t_0,t_1,\dots,t_n)$ with $a=t_0<t_1<\dots<t_n=b$ and elementary flows $f_i:[t_{i-1},t_i]\rightarrow\mathcal{A}$ ($i=1,\dots,n$) such that
\begin{equation*}
\forall s\in [t_{i-1},t_i];\quad
f(s)=f_i(s)\cup \bigcup_{j=1}^{i-1}f_j(t_j).
\end{equation*}
The set of all simple (resp. elementary) flows is denoted $S(\mathcal{A})$ (resp. $S^e(\mathcal{A})$).
\end{definition}

At  first glance, the notion of simple flow may seem artificial and unnecessary but the embedding in $\mathcal{A}(u)$ is the key point to get a characterization of distributions of set-indexed processes by projections on flows.

According to \cite{flow}, we use the parametrization of flows which allows to preserve the increment stationarity property under projection on flows (it avoids the appearance of a time-change).

\begin{definition}
For any set-indexed process $X=\left\{X_U;\;U\in\mathcal{A}\right\}$ on the space $(\mathcal{T},\mathcal{A},m)$ and any elementary flow $f:[a,b]\rightarrow\mathcal{A}$, we define the {\em $m$-standard projection of $X$ on $f$} as the process $$X^{f,m}=\left\{X^{f,m}_t=X_{f\circ\theta^{-1}(t)};\;t\in [a,b]\right\},$$
where $\theta:t\mapsto m[f(t)]$.
\end{definition}

The following result shows that Definition \ref{deflevy} for set-indexed L\'evy processes cannot be reduced to the L\'evy class for the projections on elementary flows. The increment stationarity property is characterized by the property on elementary flows, but simple flows are needed to characterize the independence of increments.

\begin{theorem}
Let $X=\left\{X_U;\;U\in\mathcal{A}\right\}$ be a set-indexed process with definite increments, then the following two assertions hold:
\begin{enumerate}[(i)]
\item\label{itcaractflow_flow} If $X$ is a set-indexed L\'evy process, then the standard projection of $X$ on any elementary flow $f:[0,T]\rightarrow\mathcal{A}$ such that $f(0)=\emptyset'$ is a real-parameter L\'evy process.

\item\label{itcaractflow_levy} If $X$ is continuous in probability, if $X_{\emptyset'}=0$, if the standard projection of $X$ on any simple flow $f:[a,b]\rightarrow\mathcal{A}(u)$ has independent increments, and if this projection has stationary increments in the special case of elementary flows,
then $X$ is a set-indexed L\'evy process.

\end{enumerate}
\end{theorem}

\begin{proof}

\begin{enumerate}[(i)]
\item According to Proposition 1.6 of \cite{Ba02} and Proposition 5.4 of \cite{flow}, if $X$ is a set-indexed L\'evy process and $f$ is an elementary flow, then the standard projection $X^{f,m}$ is a real-parameter process with independent and stationary increments.

\noindent Moreover, if $(t_n)_{n\in\N}$ is a sequence in $[0,T]$ converging to $t_{\infty}$, then the continuity of $f$ implies
\begin{align*}
\bigcap_{k\geq n} f(t_k) = f( \inf_{k\geq n}t_k ) \quad\textrm{and}\quad
\overline{\bigcup_{k\geq n} f(t_k)} = f( \sup_{k\geq n}t_k ).
\end{align*}
Then, by continuity of $f$,
\begin{align*}
\overline{\bigcup_n \bigcap_{k\geq n}f(t_k)} = \overline{\bigcup_n f( \inf_{k\geq n}t_k )} = f\left( \sup_n \inf_{k\geq n}t_k \right) = f(t_{\infty}),
\end{align*}
and
\begin{align*}
\bigcap_n \overline{\bigcup_{k\geq n}f(t_k)} = \bigcap_n f( \sup_{k\geq n}t_k ) = f\left( \inf_n \sup_{k\geq n}t_k \right) = f(t_{\infty}).
\end{align*}

\noindent From the continuity in probability of the set-indexed process $X$, we conclude that $X^{f,m}(t_n)$ converge to $X^{f,m}(t_{\infty})$ in probability. Thus $X^{f,m}$ is a real-parameter L\'evy process.

\item According to Proposition 1.6 of \cite{Ba02}, the set-indexed process $X$ has independent increments.  Proposition 5.4 of \cite{flow} implies the $m$-stationarity of $\mathcal{C}_0$-increments of $X$.

\noindent Then the continuity in probability of $X$ allows to
conclude that $X$ is a set-indexed L\'evy process.

\end{enumerate}
\end{proof}

\section{Markov properties}

The Markov property is strongly connected with L\'evy processes and has already been studied for set-indexed processes. Different authors have given various definitions for this property. Here we follow the definitions of set-Markov and $\mathcal{Q}$-Markov processes given by Balan and Ivanoff (\cite{BaIv02}), which seems to be the more appropriate in the set-indexed framework.

The notion of sub-semilattice plays an important role for the Markov property of set-indexed processes. Let us recall that a subset $\mathcal{A}'$ of $\mathcal{A}$ which is closed under arbitrary intersections is called a {\em lower sub-semilattice} of
$\mathcal{A}$. The ordering of a lower sub-semilattice $\mathcal{A}'=\{A_1,A_2, \dots\}$ is said to be {\em consistent} if $A_i\subset A_j \Rightarrow i\leq j$. 
Proceeding inductively, we can show that any lower sub-semilattice admits a consistent ordering, which is not unique in general (see \cite{BaIv02, Ivanoff}).

If $\{A_1,\dots,A_n\}$ is a consistent ordering of a finite lower sub-semilattice $\mathcal{A}'$, the set $C_i = A_i\setminus \bigcup_{j\leq i-1} A_j$ is called {\em the left neighbourhood} of $A_i$ in $\mathcal{A}'$. Since $C_i = A_i \setminus \bigcup_{A\in\mathcal{A}', A_i\nsubseteq A_i} A$, the definition of the left neighbourhood does not depend on the ordering.

Let us recall the definition of a $\mathcal{Q}$-Markov property.

\begin{definition}\label{def:Qsystem}
A collection $\mathcal{Q}$ of functions
\begin{align*}
\mathbf{R}\times\mathcal{B}(\mathbf{R})&\rightarrow \mathbf{R}_+ \\
(x, B)&\mapsto Q_{U,V}(x,B)
\end{align*}
where $U, V\in\mathcal{A}(u)$ are such that $U\subseteq V$,
is called a {\em transition system} if the following conditions are satisfied
\begin{enumerate}[(i)]
\item $Q_{U,V}(\bullet, B)$ is a random variable for all $B\in\mathcal{B}(\mathbf{R})$.
\item $Q_{U,V}(x,\bullet)$ is a probability measure for all $x\in\mathbf{R}$.
\item For all $U\in\mathcal{A}(u)$, $x\in\mathbf{R}$ and $B\in\mathcal{B}(\mathbf{R})$, $Q_{U,U}(x,B)=\delta_x(B)$.
\item\label{CK:Qsystem} For all $U\subseteq V\subseteq W\in\mathcal{A}(u)$,
\begin{align*}
\int_{\mathbf{R}} Q_{U,V}(x,dy)\ Q_{V,W}(y,B) = Q_{U,W}(x,B)
\qquad \forall x\in\R, \forall B\in\mathcal{B}(\R).
\end{align*}
\end{enumerate}
\end{definition}

\begin{definition}
A transition system $\mathcal{Q}$ is said to be {\em spatially homogeneous} if for all $U\subset V$, the function $Q_{U,V}$ satisfies
\begin{equation*}
\forall x\in\mathbf{R}, \forall B\in\mathcal{B}(\mathbf{R}),\quad
Q_{U,V}(x,B) = Q_{U,V}(0, B-x).
\end{equation*}
\end{definition}

\begin{definition}
A transition system $\mathcal{Q}$ is said to be {\em $m$-homogeneous} if the function $Q_{U,V}$ only depends on $m(V\setminus U)$, i.e. for all $U, V, U', V'$ in $\mathcal{A}(u)$ such that $U\subset V$ and $U'\subset V'$,
\begin{equation*}
m(V\setminus U)=m(V'\setminus U') \Rightarrow Q_{U,V} = Q_{U',V'}.
\end{equation*}
\end{definition}

\begin{definition}
Let $\mathcal{Q}$ be a transition system, $X=\{X_U;\;U\in\mathcal{A}\}$ a set-indexed process  with definite increments and $\left(\mathcal{F}_U\right)_{U\in\mathcal{A}}$ its minimal filtration. 
$X$ is said to be a $\mathcal{Q}$-Markov process if for all $U,V\in\mathcal{A}(u)$ with $U\subseteq V$
\begin{align*}
\forall B\in\mathcal{B}(\R),\quad
P\left( \Delta X_V \in B \mid \mathcal{F}_U \right) = Q_{U,V}(\Delta X_U, B).
\end{align*}
\end{definition}

\noindent Notice that for $U\in\mathcal{A}(u)$, $\mathcal{F}_U$ is defined by
$\displaystyle\mathcal{F}_U = \bigvee_{\genfrac{}{}{0pt}{}{V\in\mathcal{A}}{V\subseteq U}} \mathcal{F}_V$.

According to \cite{BaIv02}, $\mathcal{Q}$-Markov processes constitute a subclass of set-indexed processes satisfying the {\em set-Markov property}, i.e. such that $ \forall U\in\mathcal{A}, \forall V\in\mathcal{A}(u)$, the $\sigma$-algebras $\mathcal{F}_V$ and $\sigma(\Delta X_{U\setminus V})$ are independent conditionally to $\sigma(\Delta X_V)$.

In \cite{BaIv02}, it is proved that any set-indexed process with independent increments is a $\mathcal{Q}$-Markov process with a spatially homogeneous transition system $\mathcal{Q}$. The following result shows that the converse holds.

\begin{theorem}\label{th:MarkovIndep}
Let $X=\{X_U;\;U\in\mathcal{A}\}$ be a set-indexed process with definite increments. The two following assertions are equivalent:
\begin{enumerate}[(i)]
\item $X$ is a $\mathcal{Q}$-Markov process with a spatially homogeneous transition system $\mathcal{Q}$ ;
\item $X$ has independent increments.
\end{enumerate}
\end{theorem}

\begin{proof}
Since the implication $(ii)\Rightarrow (i)$ is proved in \cite{BaIv02}, we only need to prove the converse. We assume that $X$ is a $\mathcal{Q}$-Markov process with a spatially homogeneous transition system $\mathcal{Q}$.

The independence of increments of $X$ can be proved using projections on flows, since the $\mathcal{Q}$-Markov property and independence of increments are characterized by their analogous on simple flows (see \cite{Ba02}). Here we prefer giving a direct proof which illustrates the transition mechanism for set-indexed $\mathcal{Q}$-Markov processes.

\noindent
Consider any pairwise disjoint sets $C_1,\dots,C_n\in\mathcal{C}$. For all $1\leq i\leq n$, $C_i$ is defined by $C_i = U^{(0)}_i \setminus \left( \bigcup_{1\leq j\leq k_i} U^{(j)}_i \right)$, where $U^{(0)}_i, \dots, U^{(k_i)}_i \in\mathcal{A}$.
We define $\mathcal{A}'$ as the lower semilattice generated by the elements $U^{(j)}_i$ for all $1\leq i\leq n$ and $0\leq j\leq k_i$. We write $\mathcal{A}'=\{A_0=\emptyset',A_1,\dots,A_m\}$ with a consistent ordering.

\noindent By a reformulation of Proposition 5 (e) in \cite{BaIv02}, if $L_i$ denotes the left-neighbourhood of $A_i$ in $\mathcal{A}'$, for all Borel sets $B_0,\dots,B_m$,
\begin{align}\label{eq:MarkovC}
P&\left( \Delta X_{L_0} \in B_0, \dots, \Delta X_{L_m} \in B_m \right) \nonumber\\
&= \int_{\R^{m+1}} \mathbbm{1}_{B_0}(x_0)
\prod_{i=1}^m \mathbbm{1}_{B_i}(x_i-x_{i-1})\ Q_{\bigcup_{j=0}^{i-1}A_j, \bigcup_{j=0}^{i}A_j}(x_{i-1},dx_i) \ \mu(dx_0).
\end{align}
Since $\mathcal{Q}$ is spatially homogeneous, we get
\begin{align*}
P&\left( \Delta X_{L_0} \in B_0, \dots, \Delta X_{L_m} \in B_m \right) \\
&= \int_{\R^{m+1}} \mathbbm{1}_{B_0}(x_0)
\prod_{i=1}^m \mathbbm{1}_{B_i}(x_i-x_{i-1})\ Q_{\bigcup_{j=0}^{i-1}A_j, \bigcup_{j=0}^{i}A_j}(0,dx_i - x_{i-1}) \ \mu(dx_0) \\
&= \int_{\R^{m+1}} \mathbbm{1}_{B_0}(x_0)
\prod_{i=1}^m \mathbbm{1}_{B_i}(x_i)\ Q_{\bigcup_{j=0}^{i-1}A_j, \bigcup_{j=0}^{i}A_j}(0,dx_i) \ \mu(dx_0) \\
&= \mu(B_0) \prod_{i=1}^m Q_{\bigcup_{j=0}^{i-1}A_j, \bigcup_{j=0}^{i}A_j}(0,B_i).
\end{align*}
We deduce from this expression that $\Delta X_{L_0}, \dots, \Delta X_{L_m}$ are independent.

\noindent For all $1\leq i\leq n$, $C_i$ is the disjoint union of elements in $\{L_0,\dots,L_m\}$, then $\Delta X_{C_i}$ is the sum of some elements in $\{\Delta X_{L_0},\dots,\Delta X_{L_m}\}$. Since the $C_i$'s are pairwise disjoints, the independence of $\Delta X_{C_1},\dots,\Delta X_{C_n}$ follows from independence of $\Delta X_{L_0}, \dots, \Delta X_{L_m}$.
\end{proof}

The following result shows that set-indexed L\'evy processes constitute a sub-class of the $\mathcal{Q}$-Markov processes. As in the real-parameter case, they are characterized by the homogeneity of the transition system.

\begin{theorem}\label{th:LevyMarkov}
Let $X=\{X_U;\;U\in\mathcal{A}\}$ be a set-indexed process with definite increments. The two following assertions are equivalent:
\begin{enumerate}[(i)]
\item $X$ is a set-indexed L\'evy process ;
\item $X$ is a $\mathcal{Q}$-Markov process such that $X_{\emptyset'}=0$ and the transition system $\mathcal{Q}$ is spatially homogeneous and $m$-homogeneous.
\end{enumerate}
Consequently, if $\mathcal{Q}$ is a transition system which is both spatially homogeneous and $m$-homogeneous, then there exists a set-indexed process $X$ which is a $\mathcal{Q}$-Markov process.
\end{theorem}

\begin{proof}
In   the entire proof, we assume the existence of $U_0\in\mathcal{A}$ such that $m(U_0)=1$. If not, we consider $U_0\in\mathcal{A}$ such that $m(U_0)>0$ and we substitute $m(\bullet)$ with $m(\bullet)/m(U_0)$.

\bigskip
Suppose that $X=\{X_U;\;U\in\mathcal{A}\}$ is a set-indexed L\'evy process. In the proof of Theorem \ref{thRepCanon}, we showed that for all $C_0,\dots,C_n\in\mathcal{C}$ and all Borel sets $B_0,\dots,B_n$,
\begin{align}\label{eq:fddXC}
P&\left( \Delta X_{C_0}\in B_0, \dots,  \Delta X_{C_n}\in B_n \right) \nonumber\\
&=
\int_{\R^{n+1}} \mu^{m(\cap^{(n)})}(d\xi^{(n)}) \prod_{j_1<\dots<j_{n-1}} \mu^{m(\cap^{(n-1)}_{j_1,\dots,j_{n-1}}\setminus \cap^{(n)})}(d\xi^{(n-1)}_{j_1,\dots,j_{n-1}})
\dots \prod_{i<j}\mu^{m(\cap^{(2)}_{i,j}\setminus \cap^{(3)})}(d\xi^{(2)}_{i,j}) \nonumber\\
&\quad\times\prod_i \mu^{m(C_i\setminus
\cap^{(2)})}(d\xi^{(1)}_{i}).
 \mathbbm{1}_{B_i}\Bigg\{\xi^{(n)}+\sum_{k=1}^{n-1}
 \sum_{ \genfrac{}{}{0pt}{}{j_1<\dots<j_k}{i\in\{j_1,\dots,j_k\}} }
 \xi^{(k)}_{j_1,\dots,j_k} \Bigg\},
\end{align}
where $\mu = P_{X_{U_0}}$ with $m(U_0)=1$, and
\begin{align*}
\cap^{(2)}_{i,j} = C_i\cap C_j ;\quad
&\cap^{(3)}_{i,j,k} = C_i\cap C_j\cap C_k ;\quad
\dots \quad
\cap^{(n)} = C_1\cap\dots\cap C_n ;\\
\cap^{(2)} = \bigcup_{i<j}\cap^{(2)}_{i,j} ;\quad
&\cap^{(3)} = \bigcup_{i<j<k} \cap^{(3)}_{i,j,k} ;\quad
\dots
\end{align*}

\noindent For any lower semilattice $\mathcal{A}'=\{ A_0=\emptyset', A_1, \dots, A_k \}$ with a consistent ordering, the previous formula can be applied to the left-neighbourhoods $L_0,\dots,L_k$ of $\mathcal{A}'$. Obviously, the $L_i$ are pairwise disjoint and then
\begin{align*}
P\left( \Delta X_{L_0}\in B_0, \dots,  \Delta X_{L_n}\in B_n \right)
= \int_{\R^{n+1}} \prod_i \mu^{m(L_i)}(d\xi^{(1)}_{i})\mathbbm{1}_{B_i}( \xi^{(1)}_{i} ).
\end{align*}
Let us define the collection of functions $\mathcal{Q}$
\begin{align*}
\mathbf{R}\times\mathcal{B}(\mathbf{R})&\rightarrow \mathbf{R}_+ \\
(x, B)&\mapsto Q_{U,V}(x,B) = \mu^{m(V\setminus U)}(B-x)
\end{align*}
where $U, V\in\mathcal{A}(u)$ are such that $U\subseteq V$. We observe that $\mathcal{Q}$ is a transition system which is both spatially homogeneous and $m$-homogeneous and
\begin{align*}
P&\left( \Delta X_{L_0} \in B_0, \dots, \Delta X_{L_k} \in B_k \right) \\
&= \int_{\R^{k+1}} \mathbbm{1}_{B_0}(x_0)
\prod_{i=1}^k \mathbbm{1}_{B_i}(x_i-x_{i-1})\ Q_{\bigcup_{j=0}^{i-1}A_j, \bigcup_{j=0}^{i}A_j}(x_{i-1},dx_i) \ \mu^{m(\emptyset')}(dx_0).
\end{align*}
Then Proposition 5 (e) of \cite{BaIv02} allows to conclude that $X$ is a $\mathcal{Q}$-Markov process.

\bigskip
Conversely, assume that $\mathcal{Q}$ is a given transition system which is both spatially homogeneous and $m$-homogeneous.
For all $U\subset V$ in $\mathcal{A}(u)$ and $(x,B)\in \R\times\mathcal{B}(\R)$, we can write $Q_{U,V}(x,B) = \widetilde{Q}_{m(V\setminus U)}(B-x)$.

\noindent Condition (\ref{CK:Qsystem}) of Definition \ref{def:Qsystem} implies that for all $U,V,W\in\mathcal{A}(u)$ with $U\subseteq V\subseteq W$,
\begin{align*}
\forall x\in\R, \forall B\in\mathcal{B}(\R),\quad
\int \widetilde{Q}_{m(V\setminus U)}(dy)\ \widetilde{Q}_{m(W\setminus V)}(B-x-y)
= \widetilde{Q}_{m(W\setminus U)}(B-x)
\end{align*}
and thus
\begin{equation}\label{eq:convQ}
\widetilde{Q}_{m(V\setminus U)} * \widetilde{Q}_{m(W\setminus V)}
=\widetilde{Q}_{m(W\setminus U)}.
\end{equation}

\noindent Consider any $s,t\in\R_+$ such that $s\leq t$ and $s+t < m(\mathcal{T})$.
 From condition (\ref{IC:C1}) of Definition \ref{basic}, there exists $B\in\mathcal{A}(u)$ such that $s+t \leq m(B)$. Let $f:[0,1]\rightarrow\mathcal{A}(u)$ be a simple flow connecting $\emptyset$ to $B$. By continuity of the real function $\theta:u\mapsto m[f(u)]$, there exist $V,W \in\mathcal{A}(u)$ such that $U=\emptyset\subseteq V\subseteq W\subseteq B$, $m(V)=s$ and $m(W)=s+t$. Applying (\ref{eq:convQ}) to $U=\emptyset, V$ and $W$, we can state
\begin{equation}\label{eq:semigroupQ}
\forall 0\leq s\leq t \textrm{ such that } s+t < m(\mathcal{T}),\quad
\widetilde{Q}_s * \widetilde{Q}_t = \widetilde{Q}_{s+t}.
\end{equation}
Using the characteristic function $\widehat{\widetilde{Q}_u}$ of the probability measure $\widetilde{Q}_u$, expression (\ref{eq:semigroupQ}) is equivalent to
\begin{equation}\label{eq:semigroupQcharact}
\forall 0\leq s\leq t \textrm{ such that } s+t < m(\mathcal{T}),\quad
\widehat{\widetilde{Q}_s} \ \widehat{\widetilde{Q}_t} = \widehat{\widetilde{Q}_{s+t}}.
\end{equation}

\noindent It is well known that equation (\ref{eq:semigroupQcharact}) implies the existence of a function $\varphi:\R\rightarrow\C$ such that $\widehat{\widetilde{Q}_t} = \varphi^t$ for all $t< m(\mathcal{T})$.

\noindent Consider $U_0\in\mathcal{A}$ such that $m(U_0) = 1$ and the probability measure $\mu$ defined by $\mu(B) = Q_{\emptyset, U_0}(0,B) = \widetilde{Q}_{1}(B)$ for all $B\in\mathcal{B}(\R)$.
The function $\varphi$ is nothing but the characteristic function of $\mu$, and consequently
\begin{align*}
\forall t\in\R_+ \textrm{ such that } t < m(\mathcal{T}),\quad
\widetilde{Q}_{t} = \mu^t.
\end{align*}
Then the transition system $\mathcal{Q}$ is defined by
$Q_{U,V}(x,B) = \mu^{m(V\setminus U)}(B-x)$ for all $U\subset V$ and all $(x,B)\in\R\times\mathcal{B}(\R)$.

For any $C=U\setminus V\in\mathcal{C}_0$ with $U,V\in\mathcal{A}$ and $V\subset U$, we consider the lower semilattice $\mathcal{A}'$ generated by $U,V$. We use the consistent ordering $\mathcal{A}'=\{A_0=\emptyset', A_1=V, A_2=U\}$. From (\ref{eq:MarkovC}) with $B_0=B_1=\mathcal{T}$ and any Borel set $B_2$,
\begin{align}\label{eq:loiC0}
P(\Delta X_{U\setminus V}\in B_2) =  \mu^{m(U\setminus V)}(B_2).
\end{align}
Expression (\ref{eq:loiC0}) implies the stationarity condition of the equivalent definition for set-indexed L\'evy processes (Condition (3') of Proposition \ref{defLevyequiv}).

\noindent Moreover, Theorem \ref{th:MarkovIndep} implies that $X$ has independent increments.

It remains to prove the stochastic continuity in order to conclude that $X$ is a set-indexed L\'evy process.
Let $(U_n)_{n\in\mathbf{N}}$ be a sequence in $\mathcal{A}$ such that
\begin{equation*}
\overline{\bigcup_{n} \bigcap_{k\geq n} U_k} = \bigcap_n \overline{\bigcup_{k\geq n} U_k} = A\in\mathcal{A}.
\end{equation*}
In the same way as in the proof of Theorem \ref{thRepCanon}, we write for all $n\in\mathbf{N}$,
\begin{align*}
X_{U_n} - X_A = \Delta X_{U_n\setminus A} - \Delta X_{A\setminus U_n}.
\end{align*}
Therefore, the distribution of $X_{U_n} - X_A$ is the convolution product of the (independent) laws of $\Delta X_{U_n\setminus A}$ and $-\Delta X_{A\setminus U_n}$.
Then using (\ref{eq:loiC0}),
\begin{align*}
P(|X_{U_n} - X_A| > \epsilon) &= \int\int \mathbbm{1}(|x-y|>\epsilon)\  \mu^{m(U_n\setminus A)}(dx)\ \mu^{m(A\setminus U_n)}(dy).
\end{align*}
Since $\lim_{n\rightarrow\infty} m(U_n\setminus A)=0$ and $\lim_{n\rightarrow\infty} m(A\setminus U_n)=0$, we deduce that
\begin{align*}
P(|X_{U_n} - X_A| > \epsilon) \rightarrow 0\quad\textrm{as }n\rightarrow\infty.
\end{align*}

\

The existence of a $\mathcal{Q}$-Markov process, if $\mathcal{Q}$ is a spatially homogeneous and $m$-homogeneous transition system, follows from Theorem \ref{thRepCanon}.
\end{proof}

In \cite{BaIv02}, the existence of a $\mathcal{Q}$-Markov process was proved for a transition system $\mathcal{Q}$ which satisfies a symmetry condition: For all Borel sets $B_0,\dots, B_n$, the quantity
\begin{align*}
\int_{\R^{m+1}} \mathbbm{1}_{B_0}(x_0)
\prod_{i=1}^m \mathbbm{1}_{B_i}(x_i-x_{i-1})\ Q_{\bigcup_{j=0}^{i-1}A_j, \bigcup_{j=0}^{i}A_j}(x_{i-1},dx_i) \ \mu^{m(\emptyset')}(dx_0)
\end{align*}
does not depend on the choice of the consistent ordering $\{A_0=\emptyset',A_1,\dots,A_n\}$ of any lower semilattice $\mathcal{A}'$ (Theorem 1 and Assumption 1).

In Theorem \ref{th:LevyMarkov}, the existence is proved without any symmetry assumption on $\mathcal{Q}$. It relies on the construction theorem of set-indexed L\'evy processes (Theorem \ref{thRepCanon}), where the $m$-stationarity and independence of increments allow to define directly the finite dimensional distributions of the increment process $\{\Delta X_C;\;C\in\mathcal{C}\}$. 
The $\mathcal{A}$-indexed process $X$ is then the restriction of the additive process $\Delta X$ to $\mathcal{A}\subset\mathcal{C}$.

\section{Sample paths and semimartingale properties}

In this section, we study the sample paths of set-indexed L\'evy processes and we prove another characterization of set-indexed L\'evy processes as the sum of a martingale and a finite variation process.

We will not discuss here the measurability problems for sample paths of processes. Since the indexing collection $\mathcal{A}$ satisfies condition (4) ({\em Separability from above}) in Definition \ref{basic}, we assume that all our processes are separable.

\

In the real-parameter case, the fact that every L\'evy process is a semi-martingale comes from the decomposition of the process into the sum of a linear function, a Brownian motion and a pure jump process.
In some classical reference book on L\'evy processes (see \cite{Applebaum, Bertoin} for instance), the so-called L\'evy-It\^o decomposition implies the L\'evy-Khintchine representation.
In \cite{sato}, the L\'evy-Khintchine representation comes directly from infinitely divisible distributions and it is used to get the L\'evy-It\^o decomposition. Here, we follow this construction in the set-indexed setting.

\

In contrast to the real-parameter (and also multiparameter) setting, it is illusory to imagine a decomposition of the set-indexed L\'evy process in a continuous (Gaussian) part, and a pure jump (Poissonian) part. Indeed, even the set-indexed Brownian motion can be not continuous for some indexing collection (see \cite{Adler, AT}). 
In the general case, there can be many reasons for which a set-indexed function is discontinuous. However, in the special case of set-indexed L\'evy processes, a weaker form of the continuity property can be considered to study the sample paths.
Following the definition of \cite{AMSW} in the multiparameter setting, we will only consider a single type of discontinuity: the point mass jumps.

In this section, we assume that $\mathring{U}\neq\emptyset$ for all $U\in\mathcal{A}$, and that the collection $\mathcal{C}^{\ell}(\mathcal{A}_n)$ of the left-neighborhoods of $\mathcal{A}_n$ is a dissecting system (see \cite{Ivanoff}), i.e. for any $s,t\in\mathcal{T}$ with $s\neq t$, there exist $C$ and $C'$ in some $\mathcal{C}^{\ell}(\mathcal{A}_{n})$ such that $s\in C$, $t\in C'$ and $C\cap C'=\emptyset$.

\begin{definition}
The {\em point mass jump} of a set-indexed function
$x:\mathcal{A}\rightarrow\mathbf{R}$ at $t\in\mathcal{T}$ is
defined by
\begin{equation}\label{eq:ptmass}
J_t(x) = \lim_{n\rightarrow\infty} \Delta x_{C_n(t)}, \quad\textrm{where}\quad
C_n(t) = \bigcap_{\genfrac{}{}{0pt}{}{C\in\mathcal{C}_n}{t\in C}} C
\end{equation}
and for each $n\geq 1$, $\mathcal{C}_n$ denotes the collection of subsets $U\setminus V$ with $U\in\mathcal{A}_n$ and $V\in\mathcal{A}_n(u)$.
\end{definition}

Rigorously, a direct transposition of the definition of
\cite{AMSW} to the set-indexed framework should have led to $J_t(x) = \Delta x_{C(t)}$, where $\displaystyle C(t) = \bigcap_{t\in C\in\mathcal{C}} C$. 
However, since $C(t)$ is the difference between an element of $\mathcal{A}$ and a (possibly infinite) union of elements of $\mathcal{A}$, $C(t)\notin\mathcal{C}$ and $\Delta x_{C(t)}$ cannot be defined directly.

\begin{definition}[Pointwise continuity]
A set-indexed function $x:\mathcal{A}\rightarrow\mathbf{R}$ is said {\em pointwise-continuous} if $J_t(x) = 0$, for all $t\in\mathcal{T}$.
\end{definition}

\begin{theorem}\label{th:sampleGaussian}
Let $\{X_U;\;U\in\mathcal{A}\}$ be a set-indexed L\'evy process with Gaussian increments. Then for any $U_{\max}\in\mathcal{A}$ such that $m(U_{\max})<+\infty$, the sample paths of $X$ are almost surely pointwise-continuous inside $U_{\max}$, i.e.
\begin{equation*}
P( \forall t\in U_{\max}, J_t(X) = 0 ) = 1.
\end{equation*}
\end{theorem}

\begin{proof}
We will consider here that for all $U\in\mathcal{A}$, we have $U\subset U_{\max}$ (it suffices to restrict the indexing collection to $\{ U\cap U_{\max}, U\in\mathcal{A} \}$).

Let us consider $S_n = \sup\{ \left|\Delta X_{C_n(t)}\right|; t\in
U_{\max} \}$, where $C_n(t)$ is defined in (\ref{eq:ptmass}).
Notice that since $\mathcal{C}_n$ is closed under intersections,
the supremum is taken over ``indivisible" elements of
$\mathcal{C}_n$. These elements constitute precisely the
collection $\mathcal{C}^{\ell}(\mathcal{A}_n)$ of the
left-neighborhoods of $\mathcal{A}_n$ (see \cite{Ivanoff}). Then
the quantity $S_n$ can be rewritten as
\begin{equation*}
S_n = \sup\{ \left|\Delta X_{C}\right|; C\in \mathcal{C}^{\ell}(\mathcal{A}_n) \}.
\end{equation*}
Since $\mathcal{C}^{\ell}(\mathcal{A}_n)$ is a dissecting system
(see \cite{Ivanoff} or \cite{Censoring}) and  the measure $m$ does
not charge points, we remark that
\begin{equation}\label{eq:cvleftneighbor}
\sup_{C\in \mathcal{C}^{\ell}(\mathcal{A}_n)} m(C) \rightarrow 0
\quad\textrm{as}\quad
n\rightarrow +\infty.
\end{equation}

\noindent
For any fixed $\epsilon >0$,
\begin{align}\label{eq:majSn}
P(S_n>\epsilon) = P\left( \bigcup_{C\in \mathcal{C}^{\ell}(\mathcal{A}_n)} \left\{ \left|\Delta X_C\right| > \epsilon \right\} \right)
\leq \sum_{C\in \mathcal{C}^{\ell}(\mathcal{A}_n)} P( \left|\Delta X_C\right| > \epsilon ).
\end{align}
By hypothesis, $\Delta X_C$ is a Gaussian random variable for all
$C\in \mathcal{C}^{\ell}(\mathcal{A}_n)$. Then the
L\'evy-Khintchine characterization gives
\begin{equation*}
\forall z\in\mathbf{R},\quad
E\left[ e^{i z \Delta X_C} \right] = \exp\left\{ m(C) \left[ -\frac{1}{2}\sigma^2 z^2 + i \gamma z \right] \right\},
\end{equation*}
and therefore,
\begin{equation*}
E[ \Delta X_C ] = \gamma. m(C), \quad \var( \Delta X_C ) = \sigma^2. m(C).
\end{equation*}
Hence, for all integer $p\geq 1$, there exists a real constant $C_p>0$ such that
\begin{align*}
P\left( \left| \Delta X_C - E[ \Delta X_C ] \right| > \epsilon/2 \right)
\leq C_p\ \frac{[ \var(\Delta X_C) ]^p}{(\epsilon/2)^{2p}}
= C_p\ \frac{\sigma^{2p}}{(\epsilon/2)^{2p}}\ [ m(C) ]^p,
\end{align*}
and thus
\begin{align*}
P\left( \left| \Delta X_C \right| > \epsilon/2 + \left| E[ \Delta X_C ] \right| \right)
\leq C_p\ \frac{\sigma^{2p}}{(\epsilon/2)^{2p}}\ [ m(C) ]^p.
\end{align*}
{}From (\ref{eq:cvleftneighbor}), $| E[\Delta X_C] | < \epsilon/2$
for $n$ sufficiently great and then
\begin{align}\label{eq:majdeltagauss}
P\left( \left| \Delta X_C \right| > \epsilon \right)
\leq C_p\ \frac{\sigma^{2p}}{(\epsilon/2)^{2p}}\ [ m(C) ]^p.
\end{align}
{}From (\ref{eq:majSn}) and (\ref{eq:majdeltagauss}), we get
\begin{align*}
P(S_n>\epsilon) &\leq C_p\ \frac{\sigma^{2p}}{(\epsilon/2)^{2p}}
\sum_{C\in \mathcal{C}^{\ell}(\mathcal{A}_n)} [ m(C) ]^p \\
&\leq C_p\ \frac{\sigma^{2p}}{(\epsilon/2)^{2p}} \Bigg( \sum_{C\in
\mathcal{C}^{\ell}(\mathcal{A}_n)} m(C) \Bigg)
\sup_{C\in \mathcal{C}^{\ell}(\mathcal{A}_n)} [m(C)]^{p-1} \\
&\leq C_p\ \frac{\sigma^{2p}}{(\epsilon/2)^{2p}}\ m(U_{\max}) \
\sup_{C\in \mathcal{C}^{\ell}(\mathcal{A}_n)} [m(C)]^{p-1},
\end{align*}
using the fact that the left-neighborhoods are disjoints (see \cite{Censoring}).

\noindent From (\ref{eq:cvleftneighbor}), let us consider an extracting function $\varphi:\N\rightarrow\N$ such that
$$\sup_{C\in \mathcal{C}^{\ell}(\mathcal{A}_n)} m(C) \leq 2^{-n}$$ and
take $p=2$ in the previous inequality. The Borel-Cantelli Lemma
implies that $S_n$ converges to $0$ almost surely as
$n\rightarrow\infty$.
\end{proof}

Let notice that Theorem \ref{th:sampleGaussian} implies that set-indexed Brownian motion is almost surely pointwise-continuous for any indexed collection (even for a collection which makes it not continuous).

%\noindent
%The converse result is also true, and is a consequence of Theorem \ref{th:Levy-Ito}.

In the sequel, we study the point mass jumps of a set-indexed L\'evy process and we prove that they determine the L\'evy measure of the process.

\

\noindent Following \cite{Censoring}, we consider
$ \displaystyle A_t = \bigcap_{t\in U\in\mathcal{A}} U$ for all $t\in\mathcal{T}$, and the partial order of $\mathcal{T}$ defined by
$$ \forall s,t\in\mathcal{T},\quad s \preccurlyeq t \Leftrightarrow A_s\subseteq A_t.$$
Obviously, we can write $A_t = \{ s\in\mathcal{T}: s\preccurlyeq t \}$ and it can be proved that $ \left[s = t \Leftrightarrow A_s = A_t\right]$. This implies
$$ \forall s,t\in\mathcal{T},\quad s \prec t \Leftrightarrow A_s\subset A_t.$$
For all $a,b\in\mathcal{T}$, we define the intervals
\begin{align*}
[a,b] &= \{t\in\mathcal{T}: a\preccurlyeq t\preccurlyeq b\}\quad\textrm{and}\\
(a,b) &= \{t\in\mathcal{T}: a\prec t\prec b\} = [a,b] \setminus\{a,b\}.
\end{align*}

\

\begin{definition}
A set-indexed function $x:\mathcal{A}\rightarrow\R$ is said to satisfy the $\mathcal{C}(u)$-ILOL property (Inner Limits and Outer Limits), if it admits an extension $\Delta x$ on $\mathcal{C}(u)$ for which for any $t\in\mathcal{T}$, there exist two real numbers $\underline{L}$ and $\overline{L}$ such that:

\noindent $\forall \epsilon>0$, there exist $\delta_t>0$ and $\eta_t>0$ such that
\begin{equation}\label{def:IL}
\forall V\in\mathcal{C}(u)\textrm{ with } V\subset A_t\setminus \{t\},\quad
m(A_t\setminus V) < \delta_t \Rightarrow \left| \Delta x_V - \underline{L} \right| < \epsilon,
\end{equation}
and
\begin{equation}\label{def:OL}
\forall W\in\mathcal{C}(u)\textrm{ with } A_t\subset W,\quad
m(W\setminus A_t) < \eta_t \Rightarrow \left| \Delta x_W - \overline{L} \right| < \epsilon.
\end{equation}
We denote $\Delta x_{A_t-} = \underline{L}$ and $\Delta x_{A_t+} = \overline{L}$.
\end{definition}

In the sequel, we will consider set-indexed L\'evy processes whose sample paths satisfy the $\mathcal{C}(u)$-ILOL property. We study their point mass jumps and we prove that they admit a L\'evy-It\^o decomposition.

By $L^2$-continuity, the sample paths of the set-indexed Brownian motion satisfy the $\mathcal{C}(u)$-ILOL property almost surely.
Since the compound Poisson process only jumps on single points, we deduce that it also satisfies the $\mathcal{C}(u)$-ILOL property.

\begin{proposition}\label{prop:jumps}
Any set-indexed function $x:\mathcal{A}\rightarrow\R$ satisfying the $\mathcal{C}(u)$-ILOL property admits point mass jumps at every point, i.e. $J_t(x)$ is defined for all $t\in\mathcal{T}$.
Moreover, for any $\epsilon >0$ and any $U_{\max}\in\mathcal{A}$, the number of points $t\in U_{\max}$ such that $\left|J_t(x)\right| > \epsilon$ is finite.
\end{proposition}

\begin{proof}
For any $t\in\mathcal{T}$, condition (\ref{def:IL}) of the $\mathcal{C}(u)$-ILOL property with $V=A_t \setminus C_n(t)$ implies that for all $\epsilon>0$, there exists $\delta_t>0$ such that
\begin{equation*}
m(C_n(t)\cap A_t) < \delta_t \Rightarrow \left| \Delta x_{C_n(t)} - \Delta x_{A_t-} \right| <\epsilon.
\end{equation*}
Since the collection
$\mathcal{C}^{\ell}\left(\mathcal{A}_n\right)$ is a dissecting
system and that the measure $m$ does not charge points,
$m(C_n(t))$ converges to $0$ as $n$ goes to $\infty$. Then $\Delta
x_{C_n(t)}$ tends to $\Delta x_{A_t-}$ as as $n$ goes to $\infty$
and $J_t(x)$ is well-defined.

\

\noindent Let us define the oscillation of $x$ in $C\in\mathcal{C}$
$$ w_x(C) = \sup_{C' \subseteq C}\left| \Delta x(C') \right|.$$
As in the proof of Theorem \ref{th:sampleGaussian}, we can assume that all $U\in\mathcal{A}$ is included in $U_{\max}$.

\

\noindent For any given $\epsilon > 0$, we will show that $U_{\max}$ can be covered such a way
\begin{equation*}
U_{\max} \subset \bigcup_{1\leq i\leq k} (a_i,b_i),
\quad \textrm{with } t_i\in\mathring{\widehat{(a_i,b_i)}},
\end{equation*}
such that $w_x((a_i,t_i)) < \epsilon$ and $w_x((a_i,b_i) \setminus (a_i,t_i]) < \epsilon$.
This assertion implies that the only points of $U_{\max}$ where point mass jump can be bigger than $\epsilon$ are the $a_i$'s, $t_i$'s and $b_i$'s. Therefore their number is finite and the result follows.

\

\noindent For all $t\in U_{\max}$, the $\mathcal{C}(u)$-ILOL property implies
the existence of $\delta_t>0$ and $\eta_t>0$ such that
$$ \forall V\in\mathcal{C}(u)\textrm{ s.t. } V\subset A_t\setminus \{t\},\quad
m(A_t\setminus V) < \delta_t \Rightarrow \left| \Delta x_V - \Delta x_{A_t-} \right| < \epsilon/2,$$
and
$$ \forall W\in\mathcal{C}(u)\textrm{ s.t. } A_t\subset W,\quad
m(W\setminus A_t) < \eta_t \Rightarrow \left| \Delta x_W - \Delta x_{A_t+} \right| < \epsilon/2.$$

\noindent There exist $V_t=\{u\in\mathcal{T}: u\preccurlyeq a_t\}$ and $W_t=\{u\in\mathcal{T}: u\preccurlyeq b_t\}$ in $\mathcal{A}$ such that $V_t\subset \mathring{A_t}$ with $m(A_t\setminus V_t)<\delta_t$ and $A_t\subset \mathring{W_t}$ with $m(W_t\setminus A_t)<\eta_t$.
Since $t\in\mathring{\widehat{(a_t,b_t)}}$, a compacity argument implies
\begin{align*}
U_{\max} \subset \bigcup_{1\leq i\leq k} \mathring{\widehat{(a_{t_i},b_{t_i})}}
%\subset \bigcup_{1\leq i\leq k} \left[ (a_{t^{(i)}},t^{(i)}] \cup ( W_{t^{(i)}} \setminus C_n(t^{(i)}) ) \right]
.
\end{align*}
For each $i=1,\dots,k$, we split the interval $(a_{t_i},b_{t_i})$ in $(a_{t_i},t_i]\cup ((a_{t_i},b_{t_i})\setminus (a_{t_i},t_i])$,
\begin{itemize}
\item For any $C\in\mathcal{C}$ such that $C\subseteq (a_{t_i},t_i)$, we have $V_{t_i}\cup C\subset A_{t_i}\setminus\{t_i\}$, $V_{t_i}\cap C = \emptyset$
and $m\left(A_{t_i}\setminus (V_{t_i}\cup C)\right)<\delta_{t_i}$. \\
Then $\Delta x_{V_{t_i}\cup C} = \Delta x_{V_{t_i}}+\Delta x_{C}$ and
$\left| \Delta x_{C} \right| < \epsilon/2 + \left| \Delta x_{V_{t_i}} - \Delta x_{A_{t_i}-} \right| <\epsilon$.
This implies $w_x((a_{t_i},t_i)) < \epsilon$.

\item For any $C\in\mathcal{C}$ such that $C\subseteq (a_{t_i},b_{t_i})\setminus (a_{t_i},t_i]$, we have $C\subset W_{t_i}$, $A_{t_i}\subset W_{t_i}\setminus C$ and $m\left( (W_{t_i} \setminus C)\setminus A_{t_i} \right) < \eta_{t_i}$. \\
Then $\Delta x_{W_{t_i} \setminus C} = \Delta x_{W_{t_i}} - \Delta x_{C}$ and
$\left| \Delta x_{C} \right| < \epsilon/2 + \left| \Delta x_{W_{t_i}} - \Delta x_{A_{t_i}+} \right| <\epsilon$.
This implies $w_x((a_i,b_i) \setminus (a_i,t_i]) < \epsilon$.
\end{itemize}
\end{proof}

\begin{remark}
\noindent The proof of Proposition \ref{prop:jumps} shows that condition ``$V\subset A_t\setminus\{t\}$'' in $(\ref{def:IL})$ of $\mathcal{C}(u)$-ILOL property is essential to authorize a positive point mass jump at $t$.
If this condition is substituted with ``$V\subset A_t$'',
for any $C\in\mathcal{C}$ with $C\subset (a_{t},b_{t})$, where $a_t$ and $b_t$ are defined in the proof, we have
\begin{itemize}
\item $C\cap A_t \subset A_t\setminus V_t$ and $m\left(A_t\setminus (V_t\cup (C\cap A_t))\right)<\delta_t$. \\
Then $\Delta x_{V_t\cup (C\cap A_t)} = \Delta x_{V_t}+\Delta x_{C\cap A_t}$ and
$\left| \Delta x_{V_t} + \Delta x_{C\cap A_t} - \Delta x_{A_t-} \right| <\epsilon/2$.

\item $C\setminus A_t \subset W_t\setminus A_t$ and $m\left( (W_t \setminus (C\setminus A_t)) \setminus A_t \right) < \eta_t$. \\
Then $\Delta x_{W_t \setminus (C\setminus A_t))} = \Delta x_{W_t} - \Delta x_{C\setminus A_t}$ and
$\left| \Delta x_{W_t} - \Delta x_{C\setminus A_t} - \Delta x_{A_t+} \right| <\epsilon/2$.
\end{itemize}
Since $\Delta x_C = \Delta x_{C\cap A_t} + \Delta x_{C\setminus A_t}$, we get
\begin{align*}
\left| \Delta x_C \right| < \epsilon + \left| \Delta x_{V_t} -  \Delta x_{A_t-} \right| + \left| \Delta x_{W_t} -  \Delta x_{A_t+} \right| < 2\epsilon.
\end{align*}
Therefore, $w_x((a_t, b_t)) < 2\epsilon$ for all $\epsilon>0$, and consequently $J_t(x)=0$.
\end{remark}

As in the classical case of real parameter L\'evy processes, we consider the $\sigma$-field $\mathcal{B}_{\epsilon}$, generated by the opened subsets of $\{x\in\R: |x|>\epsilon\}$.\\
Let $X=\{X_U;\;U\in\mathcal{A}\}$ be a set-indexed L\'evy process whose sample paths satisfy the $\mathcal{C}(u)$-ILOL property, and $U_{\max}\in\mathcal{A}$.
Recall that the conditions on the L\'evy measure $\nu$ of $X$  implies that $\nu(B)<+\infty$ for all $B\in\mathcal{B}_{\epsilon}$.\\
According to Proposition \ref{prop:jumps}, the $\mathcal{C}(u)$-ILOL property insures that the number of point mass jumps inside $\mathcal{B}_{\epsilon}$ is finite. 
Then, we can define, for all $U\in\mathcal{A}$ with $U\subset U_{\max}$, 
\begin{align}\label{eq:NU}
N_U(B) &= \#\left\{ t\in U: J_t(X) \in B \right\}, \\
X^B_U &= \int_B x.N_U(dx) \label{eq:XB},
\end{align}
for all $B\in\mathcal{B}_{\epsilon}$, as in the multidimensional case studied in \cite{AMSW}.

\noindent We are now able to progress towards the decomposition of a set-indexed L\'evy process into a part with no point mass jumps and a Poissonian part which leans on the points where the sample path jumps.
The next step is the expression of the L\'evy measure $\nu$, which comes from the L\'evy-Khintchine representation, in terms of the counting measure of point mass jumps $N_U$.

\noindent The following Lemma \ref{lem:NU-Poisson} and Proposition \ref{prop:Poisson} can be proved by minor adaptations of analogous results proved in the multidimensional case (Propositions~4.3, 4.4 and 4.5 in \cite{AMSW}) to the set-indexed framework.
They rely on the approximation of $U\in\mathcal{A}$ by unions of elements of $\mathcal{C}^{\ell}(\mathcal{A}_n)$.

\begin{lemma}\label{lem:NU-Poisson}
For all $U\in\mathcal{A}$ with $U\subset U_{\max}$ and all $B\in\mathcal{B}_{\epsilon}$, $N_U(B)$ and $X^B_U$ are random variables.
\end{lemma}

\noindent The following result is a consequence of Theorem \ref{thRepCanon} and the L\'evy-Khinchine formula for set-indexed L\'evy processes. 
% Its proof is totally identical to the proof of Propositions~4.4 and 4.5 in \cite{AMSW}.

\begin{proposition}\label{prop:Poisson}
\begin{enumerate}[(i)]
\item For all $B\in\mathcal{B}_{\epsilon}$, $\{N_U(B);\; U\in\mathcal{A}, U\subset U_{\max}\}$ is a set-indexed homogeneous Poisson process, with mean measure given by
\begin{equation*}
E\left[N_U(B)\right] = m(U)\ \nu(B),
\end{equation*}
where $\nu$ denotes the L\'evy measure of $X$.

\noindent Moreover, if $B_1, \dots, B_n$ are pairwise disjoint elements of $\mathcal{B}_{\epsilon}$, then the processes $\{N_U(B_1);\;U\in\mathcal{A}, U\subset U_{\max}\}$,\dots, $\{N_U(B_n);\;U\in\mathcal{A}, U\subset U_{\max}\}$ are independent.

\

\item For all $B\in\mathcal{B}_{\epsilon}$, $\{X^B_U;\;U\in\mathcal{A},U\subset U_{\max}\}$ is a set-indexed compound Poisson process such that
\begin{equation*}
\log E\left[ e^{i z X^B_U} \right] = m(U)\ \int_B \left[ e^{i z x} - 1 \right].\nu(dx),
\end{equation*}
where $\nu$ denotes the L\'evy measure of $X$.

\noindent Moreover, if $B_1, \dots, B_n$ are pairwise disjoint elements of $\mathcal{B}_{\epsilon}$ and $B=\bigcup_{1\leq j\leq n}B_j$, then the processes $\{X^{B_1}_U;\;U\in\mathcal{A}, U\subset U_{\max}\}$,\dots, $\{X^{B_n}_U;\;U\in\mathcal{A}, U\subset U_{\max}\}$ and $\{X_U -X^B_U;\;U\in\mathcal{A}, U\subset U_{\max}\}$ are independent.
\end{enumerate}
\end{proposition}

Proposition \ref{prop:Poisson} constitutes the key result to derive the L\'evy-It\^o decomposition from the L\'evy-Khintchine formula. 
The decomposition in the set-indexed setting is really similar to the classical real-parameter case. However, since the notion of continuity is adherent to the choice of the indexing collection, it is hopeless to obtain a split of the set-indexed L\'evy process into a continuous part and a pure jump part. 
We observe that the process is splitted into a Gaussian part without any point mass jumps (but which can be not continuous), and a Poissonian part, whose L\'evy measure counts the point mass jumps.

\begin{theorem}[L\'evy-It\^o Decomposition]\label{th:Levy-Ito}
Let $X=\{X_U;\;U\in\mathcal{A}\}$ be a set-indexed L\'evy process whose sample paths satisfy the $\mathcal{C}(u)$-ILOL property and let $(\sigma, \gamma, \nu)$ the generating triplet of $X$.\\
Then $X$ can be decomposed as
\begin{align*}
\forall\omega\in\Omega, \forall U\in\mathcal{A}, \quad
X_U(\omega) = X^{(0)}_U(\omega) + X^{(1)}_U(\omega),
\end{align*}
where
\begin{enumerate}[(i)]
\item $X^{(0)}=\{ X^{(0)}_U;\;U\in\mathcal{A} \}$ is a set-indexed L\'evy process with Gaussian increments, with generating triplet $(\sigma, \gamma, 0)$,

\item $X^{(1)}=\{ X^{(1)}_U;\;U\in\mathcal{A} \}$ is the set-indexed L\'evy process with generating triplet $(0,0,\sigma)$, defined for some $\Omega_1\in\mathcal{F}$ with $P(\Omega_1)=1$ by
\begin{align}\label{eq:LevyItoPoisson}
\forall\omega\in\Omega_1,\  &\forall U\in\mathcal{A}, \nonumber\\
X^{(1)}_U(\omega) &= \int_{|x|>1} x\ N_U(dx,\omega)
+\lim_{\epsilon\downarrow 0}\int_{\epsilon<|x|\leq 1}x \left[N_U(dx,\omega) - m(U)\right] \nu(dx),
\end{align}
where $N_U$ is defined in (\ref{eq:NU}) and the last term of (\ref{eq:LevyItoPoisson}) converges uniformly in $U\subset U_{\max}$ (for any given $U_{\max}\in\mathcal{A}$) as $\epsilon\downarrow 0$,
\item and the processes $X^{(0)}$ and $X^{(1)}$ are independent.
\end{enumerate}

\end{theorem}

\begin{proof}
The first step is the definition of the process $X^{(1)}$ by (\ref{eq:LevyItoPoisson}).
As in the proof of Theorem 4.6 in \cite{AMSW}, Proposition \ref{prop:Poisson} and Wichura's maximal inequality (\cite{Wichura}) imply the almost sure uniform convergence and that $X^{(1)}$ is a set-indexed L\'evy process with generating triplet $(0,0,\nu)$.
We denote by $\Omega_1$ the set of convergence of the second term of (\ref{eq:LevyItoPoisson}), and we set $X^{(1)}_U(\omega)=0$ for all $\omega\in\Omega\setminus\Omega_1$ and all $U\in\mathcal{A}$ with $U\subset U_{\max}$.

\noindent Then we define, for all $\omega\in\Omega_1$,
$$ \forall U\in\mathcal{A},\quad X^{(0)}_U(\omega) = X_U(\omega) - X^{(1)}_U(\omega).$$
$X^{(0)}$ is a set-indexed L\'evy process with no point mass jumps, and independent of $X^{(1)}$ (Proposition \ref{prop:Poisson}). Its characteristic exponent gives the generating triplet $(\sigma,\gamma,0)$.
\end{proof}

As in the classical case of real-parameter L\'evy processes, the L\'evy-It\^o decomposition implies a characterization of presence of jumps in the sample paths.

\begin{corollary}
Let $X=\{X_U;\;U\in\mathcal{A}\}$ be a set-indexed L\'evy process whose sample paths satisfy the $\mathcal{C}(u)$-ILOL property and let $(\sigma, \gamma, \nu)$ the generating triplet of $X$.
Then the following assertions are equivalent:
\begin{enumerate}[(i)]
\item Almost surely the sample path of $X$ has no point mass jumps.
\item $X$ has Gaussian increments.
\item The L\'evy measure $\nu$ of $X$ is null.
\end{enumerate}
\end{corollary}

A set-indexed function $x:\mathcal{A}\rightarrow\R$ is said {\em piecewise constant} if any $U_{\max}\in\mathcal{A}$ admits a partition $U_{\max}=\bigcup_{1\leq i\leq m} C_i$
such that $C_1,\dots,C_m\in\mathcal{C}$, and the functions
\begin{align*}
\mathcal{C} &\rightarrow\R \\
C &\mapsto \Delta x_{C\cap C_i}
\end{align*}
are constant.

\begin{corollary}
Let $X=\{X_U;\;U\in\mathcal{A}\}$ be a set-indexed L\'evy process whose sample paths satisfy the $\mathcal{C}(u)$-ILOL property and let $(\sigma, \gamma, \nu)$ the generating triplet of $X$.
Then the following assertions are equivalent:
\begin{enumerate}[(i)]
\item Almost surely the sample path of $X$ is piecewise constant.
\item $X$ is a compound Poisson process or the null process.
\item The generating triplet of $X$ satisfies $\sigma=\gamma=0$ and $\nu(\R)<+\infty$.
\end{enumerate}
\end{corollary}

\

In the set-indexed framework, several definitions of martingales can be considered. We refer to \cite{Ivanoff} for a comprehensive study on them.
Here we only consider the {\em strong martingale} property: $\{X_U;\;U\in\mathcal{A}\}$ is a strong martingale if
$$ \forall C\in\mathcal{C},\quad E\left[ \Delta X_C \mid \mathcal{G}^*_C \right] = 0,$$
where $\mathcal{G}^*_C = \sigma(X_U;\; U\in\mathcal{A}, U\cap C=\emptyset)$.

\noindent The notion of strong martingale can be localized using stopping sets. A stopping set with respect to $(\mathcal{F}_U)_{U\in\mathcal{A}}$ is a function $\xi:\Omega\rightarrow\mathcal{A}(u)$ satisfying: $\{\omega: U\subseteq\xi(\omega)\} \in\mathcal{F}_U$ for all $U\in\mathcal{A}$, $\{\omega: V=\xi(\omega)\} \in\mathcal{F}_V$ for all $V\in\mathcal{A}(u)$ and there exists $W\in\mathcal{A}$ such that $\xi\subseteq W$ a.s.\\
\noindent The process $\{X_U;\;U\in\mathcal{A}\}$ is a local strong martingale if there exists an increasing sequence of stopping sets $(\xi_n)_{n\in\N}$ such that $\bigcup_{n\in\N}\mathring{\widehat{\xi_n(\omega)}} = \mathcal{T}$ for all $\omega\in\Omega$ and for all $n\in\N$, $X^{\xi_n} = \{X_{\xi_n\cap U};\;U\in\mathcal{A}\}$ is a strong martingale.

\begin{definition}
A set-indexed process $\{X_U;\;U\in\mathcal{A}\}$ is called a {\em strong semi-martingale} if it can be decomposed as
\begin{align*}
\forall U\in\mathcal{A},\quad X_U = \varphi(U) + Y_U,
\end{align*}
where $\{Y_U;\;U\in\mathcal{A}\}$ is a local strong martingale and $\varphi$ is a locally finite measure on $\mathcal{T}$.
\end{definition}

\

\begin{theorem}\label{th:semimart}
Any set-indexed L\'evy process $X=\{X_U;\;U\in\mathcal{A}\}$ whose sample paths satisfy the $\mathcal{C}(u)$-ILOL property is a strong semi-martingale.
\end{theorem}

\begin{proof}
According to the L\'evy-It\^o decomposition (Theorem \ref{th:Levy-Ito}), if $(\sigma,\gamma,\nu)$ is the generating triplet of $X$, the process can be decomposed in the sum of three terms :
\begin{itemize}
\item $X^{(0)}$ is a set-indexed L\'evy process with generating
triplet $(\sigma,\gamma,0)$. The process $\{ X^{(0)}_U -
\gamma.m(U);\;U\in\mathcal{A} \}$ is a mean zero process with
independent increments and therefore, a strong martingale (Theorem
3.4.1 in \cite{Ivanoff}) ;

\item $\displaystyle Y_U(\omega)=\int_{|x|>1} x\ N_U(dx,\omega)$ for all $\omega\in\Omega$.\\
According to Proposition \ref{prop:Poisson},
$\{Y_U;\;U\in\mathcal{A}\}$ is a set-indexed compound Poisson
process of L\'evy measure $\nu$. Then it admits a representation
$$ \forall U\in\mathcal{A};\quad Y_U = \sum_{j} X_j \mathbbm{1}_{\{\tau_j\in U\}},$$
where $(X_n)_{n\in\N}$ is a sequence of i.i.d. real random variables and $\tilde{N}_U=\sum_j\mathbbm{1}_{\{\tau_j\in U\}}$ ($U\in\mathcal{A}$) defines a set-indexed Poisson process independent of $(X_n)_{n\in\N}$, and with mean measure $\mu=\nu(\{|x|>1\}).m$.\\
For all $U\in\mathcal{A}$, we compute
\begin{align*}
E\left[Y_U\right] &= \sum_j E\left[X_j \mathbbm{1}_{\{\tau_j\in U\}}\right]
= \sum_j E[X_j]\ E\left[\mathbbm{1}_{\{\tau_j\in U\}}\right] \\
&= E[X_0]\ E\bigg[ \sum_j \mathbbm{1}_{\{\tau_j\in U\}} \bigg] =
E[X_0]\ E[\tilde{N}_U] = E[X_0]\ \mu(U).
\end{align*}

\item $\displaystyle Z_U(\omega)=\lim_{\epsilon\downarrow 0}\int_{\epsilon<|x|\leq 1}x \left[N_U(dx,\omega) - m(U) \nu(dx)\right]$ for all $\omega\in\Omega_1$ with $P(\Omega_1)=1$.
For all $0<\epsilon\leq 1$ and all $U\in\mathcal{A}$, we have
\begin{align*}
E\left[ \int_{\epsilon<|x|\leq 1}x\ N_U(dx,\omega) \right] =  m(U) \int_{\epsilon<|x|\leq 1}x\ \nu(dx).
\end{align*}
Then by $L^2$ convergence, we deduce $E[Z_U]=0$ for all $U\in\mathcal{A}$.\\
Theorem \ref{th:Levy-Ito} and Proposition \ref{prop:Poisson}, with
$\{|x|>1\} \cap \{ \epsilon<|x|\leq 1 \} = \emptyset$, imply that
$X^{(0)}$, $Y$ and $Z$ are independent. Then $Z$ is a mean zero
process with independent increments and therefore, a strong
martingale.
\end{itemize}

Aggregating the three points, we deduce that $X$ is the sum of a locally finite measure and a strong martingale.
\end{proof}

\

\bibliographystyle{plain}
%\bibliography{style}

\end{document}